\documentclass[11pt]{scrartcl}
\usepackage{amsfonts,amssymb,amsthm,amsmath}
\usepackage{color,graphicx}

\ifx\pdftexversion\undefined
\usepackage{nohyperref}
\else
\usepackage[colorlinks=true,linkcolor={blue},citecolor={blue},urlcolor={blue},
  pdfauthor={Thomas Apel, Ariel Lombardi, and Max Winkler},
  pdfstartview={Fit}]{hyperref}
\fi

\newtheorem{theorem}{Theorem}[section]
\newtheorem{corollary}[theorem]{Corollary}

\newtheorem{lemma}[theorem]{Lemma}

\newtheorem{remark}[theorem]{Remark}

\def\R{\mathbb{R}}
\def\N{\mathbb{N}}
\def\NT{\mathcal N_T}

\newcommand{\Th}{\mathcal{T}_h}
\newcommand{\Nin}{\mathcal{N}_{\mathrm{in}}}

\newcommand{\Nc}{\mathcal{N}_{\mathrm{c}}}
\newcommand{\Ns}{\mathcal{N}_{\mathrm{s}}}
\newcommand{\NN}{\mathcal{N}}
\newcommand{\node}{n}

\begin{document}
\title{Anisotropic mesh refinement in polyhedral domains:
error estimates with data in $L^2(\Omega)$}

\author{Thomas Apel\thanks{Institut f\"ur Mathematik und Bauinformatik,
    Universit\"at der Bundeswehr M\"unchen, Germany.
    \texttt{thomas.apel@unibw.de}} \and Ariel L.
  Lombardi\thanks{Departamento de Matem\'atica, Universidad de Buenos
    Aires, and Instituto de Ciencias, Universidad Nacional de General
    Sarmiento. Member of CONICET Argentina. \texttt{aldoc7@dm.uba.ar}}
  \and Max Winkler\thanks{Institut f\"ur Mathematik und Bauinformatik,
    Universit\"at der Bundeswehr M\"unchen, Germany.
    \texttt{max.winkler@unibw.de}}}

\maketitle

\begin{footnotesize}
\noindent{{\bf Abstract.} The paper is concerned with the
finite element solution of the Poisson equation with homogeneous
Dirichlet boundary condition in a three-dimensional domain.
Anisotropic, graded meshes from a former paper are reused for dealing
with the singular behaviour of the solution in the vicinity of the
non-smooth parts of the boundary. The discretization error is analyzed
for the piecewise linear approximation in the $H^1(\Omega)$- and
$L^2(\Omega)$-norms by using a new quasi-interpolation operator. This
new interpolant is introduced in order to prove the estimates for
$L^2(\Omega)$-data in the differential equation which is not possible
for the standard nodal interpolant. These new estimates allow for the
extension of certain error estimates for optimal control problems with
elliptic partial differential equation and for a simpler proof of the discrete
compactness property for edge elements of any order on this kind of
finite element meshes.}

\medskip


\medskip

\end{footnotesize}

\section{Introduction}

We consider the homogeneous Dirichlet problem for the Laplace equation,
\begin{equation}\label{eq:poissonproblem}
  -\Delta u=f \quad\mbox{in }\Omega, \qquad
  u=0 \quad\mbox{on }\partial\Omega,
\end{equation}
where $\Omega$ is a polyhedral domain. Note that we could consider a
more general elliptic equation of second order. But by a linear change
of the independent variables the main part of the differential
operator could be transformed to the Laplace operator in another
polyhedral domain such that it is sufficient to consider the Laplace
operator here.

The aim of the paper is to prove the discretization error estimate
\begin{equation}\label{eq:aim}
  \|u-u_h\|_{H^1(\Omega)}\le Ch \|f\|_{L^2(\Omega)}
\end{equation}
for the finite element solution $u_h\in V_h$ which is constructed by
using piecewise linear and continuous functions on a family of
appropriate finite element meshes $\Th$. Note that we assume here not
more than $f\in L^2(\Omega)$ such that the $L^2$-error estimate
\begin{equation}\label{eq:aim2}
  \|u-u_h\|_{L^2(\Omega)}\le Ch^2 \|f\|_{L^2(\Omega)}
\end{equation}
follows by the Aubin--Nitsche method immediately. The generic constant
$C$ may have different values on each occurrence.

If the solution of the boundary value problem
\eqref{eq:poissonproblem} was in $H^2(\Omega)$ then the finite element
meshes could be chosen quasi-uniform, and the error estimates
\eqref{eq:aim} and \eqref{eq:aim2} would be standard. However, if the
domain $\Omega$ is non-convex, the solution will in general contain
vertex and edge singularities, that means $u\not\in H^2(\Omega)$. In
this case the convergence order is reduced in comparison with
\eqref{eq:aim} and \eqref{eq:aim2} when quasi-uniform meshes are used.
As a remedy, we focus here on a priori anisotropic mesh grading
techniques as they were investigated by Apel and Nicaise in
\cite{apel:97c}. In comparison with isotropic local mesh refinement,
the use of anisotropic elements avoids an unnecessary refinement along
the edges.

The estimate \eqref{eq:aim} is in general proven by using the C\'ea
lemma (or the best approximation property of the finite element
method),
\begin{equation}\label{eq:cea}
  \|u-u_h\|_{H^1(\Omega)}\le C\inf_{v_h\in V_h}\|u-v_h\|_{H^1(\Omega)},
\end{equation}
and by proving an interpolation error estimate as an upper bound for
the right-hand side of \eqref{eq:cea}. The particular difficulty is
that when the Lagrange interpolant is used together with anisotropic
mesh grading, then the local interpolation error estimate
\begin{equation}\label{eq:locestW2p}
  |u-I_hu|_{W^{1,p}(T)}\le h_T|u|_{W^{2,p}(T)}
\end{equation}
does not hold for $p=2$ but only for $p>2$, see \cite{apel:92}. Hence
the classical proof of a finite element error estimate via
\[
  \|u-u_h\|_{H^1(\Omega)}\le C \|u-I_hu\|_{H^1(\Omega)}\le C
  \left(\sum_{T\in\Th} h_T|u|_{H^2(T)}^2\right)^{1/2}
\]
does not work. This problem was overcome by Apel and Nicaise,
\cite{apel:97c}, by using \eqref{eq:locestW2p} and related estimates in
weighted spaces, as well as the H\"older inequality for the prize that
$f\in L^p(\Omega)$ with $p>2$ has to be assumed in problem
\eqref{eq:poissonproblem}. Hence estimate \eqref{eq:aim} cannot be
proved in this way.

For prismatic domains and tensor product type meshes the problem was
overcome in \cite{apel:97b,apel:08b} by proving local estimates for a certain
quasi-interpolation operator. This work cannot be easily extended to
general polyhedral domains since the orthogonality of certain edges of
the elements was used there. The aim of the current paper is to
construct a quasi-interpolation operator $D_h$ such that the error estimate
\begin{equation}\label{eq:aim3}
  \|u-D_hu\|_{H^1(\Omega)}\le Ch \|f\|_{L^2(\Omega)}
\end{equation}
can be proved for the anisotropic meshes introduced in \cite{apel:97c}.

Quasi-interpolants were introduced by Cl\'ement \cite{clement:75}. The
idea is to replace nodal values by certain averaged values such that
non-smooth functions can be interpolated. This original idea has been
modified by many authors since then. The contribution by Scott and
Zhang \cite{scott:90} was most influential to our work.

The plan of the paper is as follows. In Section \ref{sec:regularity}
we introduce notation, recall regularity results for the solution $u$
of \eqref{eq:poissonproblem} and describe the
finite element discretization. The main results are proved in Section
\ref{sec:mainresults}. The paper continues with numerical results in
Section \ref{sec:tests} and ends with two sections where we describe
applications which motivated us to improve the approximation result
from $\|u-u_h\|_{H^1(\Omega)}\le Ch \|f\|_{L^p(\Omega)}$, $p>2$, to
$\|u-u_h\|_{H^1(\Omega)}\le Ch \|f\|_{L^2(\Omega)}$. The first one is
a discretization of a distributed optimal control problem with
\eqref{eq:poissonproblem} as the state equation. The second
application consists in a simpler proof of the discrete compactness
property for edge elements of any order on this kind of finite element
meshes.

We finish this introduction by commenting on related work. The idea to
treat singularities due to a non-smooth boundary by using graded
finite element meshes is old. The two-dimensional case was
investigated by Oganesyan and Rukhovets \cite{oganesyan:68}, Babu\v
ska \cite{babuska:70a}, Raugel \cite{raugel:diss}, and Schatz and
Wahlbin \cite{schatz:79}. In three dimensions we can distinguish
isotropic mesh grading, see the papers by Apel and Heinrich
\cite{apel:94a} and Apel, S\"andig, and Whiteman \cite{apel:93b}, and
anisotropic mesh grading, see the already mentioned papers
\cite{apel:92,apel:97b,apel:08b} for the special case of prismatic
domains, and \cite{apel:97c} for general polyhedral domains. This work
has been extended by B\u{a}cu\c{t}\u{a}, Nistor, and Zikatanov
\cite{BacutaNistorZikatanov:07} to higher order finite element
approximations where naturally higher regularity of the right-hand
side $f$ has to be assumed. Boundary element methods with anisotropic,
graded meshes have been considered by von Petersdorff and Stephan
\cite{petersdorff:90a}.
The main alternative to mesh grading is augmenting the finite element
space with singular functions, see for example Strang and Fix
\cite{strang:73}, Blum and Dobrowolski \cite{blum:82}, or Assous,
Ciarlet Jr., and Segr\'e \cite{assous:00a} for various variants. It
works well in two dimensions where the coefficient in front of the
singular function is constant. In the case of edge singularities this
coefficient is a function which can be approximated, see Beagles and
Whiteman \cite{beagles:86}, or it can be treated by Fourier analysis,
see Lubuma and Nicaise \cite{lubuma:94b}.

\setcounter{equation}{0}
\section{\label{sec:regularity}Notation, regularity, discretization}

It is well known that the solution of the boundary value problem
\eqref{eq:poissonproblem} contains edge and vertex singularities which
are characterized by singular exponents. For each edge $e$, the
corresponding leading (smallest) singular exponent $\lambda_e$ is
simply defined by $\lambda_e=\pi/\omega_e$ 
where $\omega_e$ is the
interior dihedral angle at the edge $e$. For vertices $v$ of $\Omega$,
the leading singular exponent $\lambda_v>0$ has to be computed via the
eigenvalue problem of the Laplace-Beltrami operator on the
intersection of $\Omega$ and the unit sphere centered at $v$
. Note that
$\lambda_e>\frac12$ and $\lambda_v>0$. A vertex $v$ or an edge $e$
will be called \emph{singular} if $\lambda_v<\frac12$ or
$\lambda_e<1$, respectively.  We exclude the case that $\frac12$ is a
singular exponent of any vertex. For a detailed discussion of edge and 
vertex singularities we refer to \cite[Sections 2.5 and 2.6]{grisvard:92b}.

As in \cite{apel:97c} we subdivide the domain $\Omega$ into
a finite number of disjoint tetrahedral subdomains, subsequently
called \emph{macro-elements},
\[
\overline\Omega=\bigcup_{\ell=1}^L \overline{\Lambda_\ell}.
\]
We assume that each $\Lambda_\ell$ contains at most one singular edge
and at most one singular vertex. In the case that $\Lambda_\ell$
contains both a singular edge and a singular vertex, that vertex is
contained in that edge.  Note that the edges of $\Lambda_\ell$ are
considered to  have $O(1)$ length.  For $\ell_1\ne \ell_2$, the
closures of the macroelements $\Lambda_{\ell_1}$ and
$\Lambda_{\ell_2}$ may be disjoint or they intersect defining a
\emph{coupling face}, or a \emph{coupling edge}, or a \emph{coupling
  node}.  Denote by $\mathcal F_c$, $\mathcal E_c$ and $\mathcal N_c$
the sets of coupling faces, edges and nodes, respectively.

For the description of the regularity of the solution $u$ of
\eqref{eq:poissonproblem}, we set
$\lambda_{\mathrm{v}}^{(\ell)}=\lambda_v$ if the macro-element
$\Lambda_\ell$ contains the singular vertex $v$ of $\Omega$. If
$\Lambda_\ell$ does not contain any singular vertex we set
$\lambda_{\mathrm{v}}^{(\ell)}=+\infty$. Moreover, we set
$\lambda_{\mathrm{e}}^{(\ell)}=\lambda_e$ if $\Lambda_\ell$ contains
the singular edge $e$ of $\Omega$, otherwise we set
$\lambda_{\mathrm{e}}^{(\ell)}=+\infty$. Furthermore, we define in
each macro-element $\Lambda_\ell$ a Cartesian coordinate system
$x^{(\ell)}=(x_1^{(\ell)},x_2^{(\ell)},x_3^{(\ell)})$ such that the
singular vertex, if existing, is located in the origin, and the
singular edge, if existing, is contained in the $x_3^{(\ell)}$-axis.
We also introduce by 
\begin{align*}
r^{(\ell)}(x^{(\ell)}) &:= 
\left((x_1^{(\ell)})^2+(x_2^{(\ell)})^2\right)^{1/2}, \\
R^{(\ell)}(x^{(\ell)}) &:=
\left((x_1^{(\ell)})^2+(x_2^{(\ell)})^2+(x_3^{(\ell)})^2\right)^{1/2}, \\
\theta^{(\ell)}(x^{(\ell)})&:=
\frac{r^{(\ell)}(x^{(\ell)})}{R^{(\ell)}(x^{(\ell)})},
\end{align*}
the distance to the $x_3^{(\ell)}$-axis, the distance to the origin,
the angular distance from the $x_3^{(\ell)}$-axis, respectively.

For $k\in\N$ and $\beta,\delta\in\R$ we
define the weighted Sobolev space 
\[
  V^{k,2}_{\beta,\delta}(\Lambda_\ell):=\left\{ v\in \mathcal{D}'(\Lambda_\ell):
  \|v\|_{V^{k,2}_{\beta,\delta}(\Lambda_\ell)} < \infty \right\}
\]
where
\begin{align*}
  \|v\|_{V^{k,2}_{\beta,\delta}(\Lambda_\ell)}^2 &:= \sum_{|\alpha|\le k}
  \int_{\Lambda_\ell} \left|R^{\beta-k+|\alpha|} \theta^{\delta-k+|\alpha|}
  D^\alpha v \right|^2, \\
  |v|_{V^{k,2}_{\beta,\delta}(\Lambda_\ell)}^2 &:= \sum_{|\alpha|= k}
  \int_{\Lambda_\ell} \left|R^{\beta} \theta^{\delta}  D^\alpha v \right|^2
\end{align*}
Here, we have used the standard multi-index notation to describe
partial derivatives, and we have omitted the index $(\ell)$ in $R$ and
$\theta$ for simplicity.

\begin{theorem}[\textbf{Regularity}] \label{thm:regularity}
  \cite[Theorem 2.10]{apel:97c} The weak solution $u$ of the boundary
  value problem \eqref{eq:poissonproblem} admits the decomposition
  \[ u = u_{\mathrm{r}} + u_{\mathrm{s}} \]
  in $\Lambda_\ell$, $\ell=1,\ldots,L$,
  where $u_{\mathrm{r}}\in H^2(\Lambda_\ell)$ and
  \[
    \frac{\partial u_{\mathrm{s}}}{\partial x_i^{(\ell)}} 
    \in V^{1,2}_{\beta,\delta}(\Lambda_\ell), \quad i=1,2,\qquad 
    \frac{\partial u_{\mathrm{s}}}{\partial x_3^{(\ell)}} 
    \in V^{1,2}_{\beta,0}(\Lambda_\ell),
  \]
  for any $\beta,\delta\ge0$ satisfying
  $\beta>\frac12-\lambda_{\mathrm{v}}^{(\ell)}$ and
  $\delta>1-\lambda_{\mathrm{e}}^{(\ell)}$.
\end{theorem}

Following \cite{apel:97c} we consider a triangulation $\Th$ of $\Omega$,
\[
  \overline\Omega=\bigcup_{T\in\Th} \overline T,
\]
made up of tetrahedra which match the initial partition: if
$T\cap\Lambda_\ell\not=\emptyset$ then $T\subset\Lambda_\ell$.  Four
cases are considered:

\begin{enumerate}
\item If $\Lambda_\ell$ does neither contain a singular edge nor a
  singular vertex then $\Th|_{\Lambda_\ell}$ is assumed to be
  isotropic and quasi-uniform with element size $h$, see Figure
  \ref{fig:macros}, top left.

\item If $\Lambda_\ell$ contains a singular vertex but no singular
  edges then $\Th|_{\Lambda_\ell}$ is isotropic and has a singular
  vertex refinement, i.e., the mesh is graded towards the singular
  vertex with a grading parameter~$\nu_\ell\in(0,1]$. This can be achieved by
  using a coordinate transformation of the vertices from Case 1, see
  Figure \ref{fig:macros}, top right.

\item If $\Lambda_\ell$ contains a singular edge but no singular
  vertices then $\Th|_{\Lambda_\ell}$ is anisotropically graded
  towards the singular edge. The grading parameter is $\mu_\ell\in(0,1]$.  To
  this end, we introduce a family $\mathcal P_\ell$ of planes
  transversal to the singular edge and containing the opposite one.
  These planes split the macro element into strips and contain all
  nodes. In the planes the position of the nodes is achieved by
  applying a coordinate transformation to a uniform triangulation, see
  Figure \ref{fig:macros}, bottom left.

\item If $\Lambda_\ell$ contains both a singular vertex and a singular
  edge then $\Th|_{\Lambda_\ell}$ is graded towards the singular edge
  with grading parameter $\mu_\ell\in(0,1]$ and towards the singular vertex
  with grading parameter~$\nu_\ell\in(0,1]$. The mesh is topologically
  equivalent to the mesh of Case 3 but the planes of $\mathcal P_\ell$ do
  not divide the singular edge equidistantly but with a grading towards
  the singular vertex.
\end{enumerate}

\begin{figure}
  \begin{center}
    \includegraphics[width=40mm]{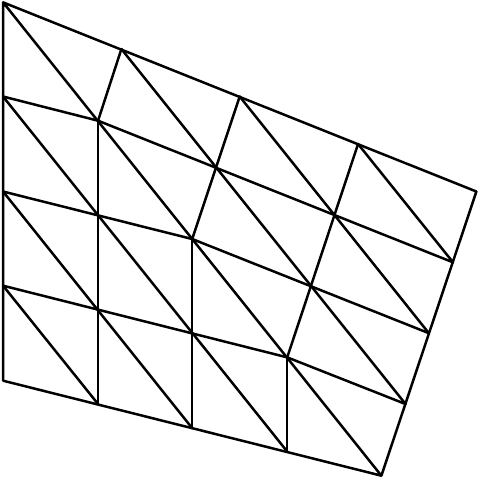} \qquad 
    \includegraphics[width=40mm]{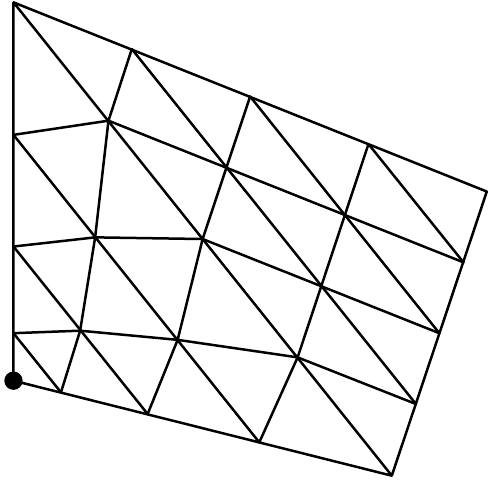} \\
    \includegraphics[width=40mm]{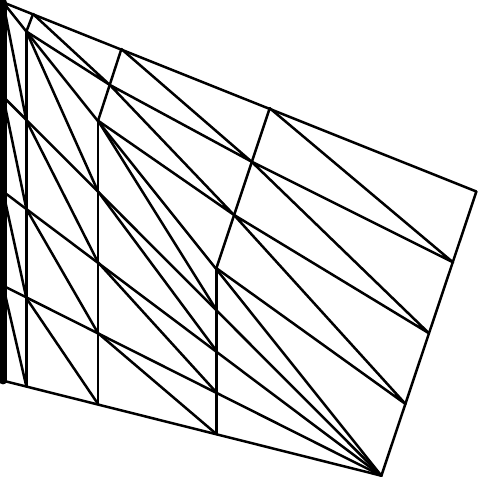} \qquad
    \includegraphics[width=40mm]{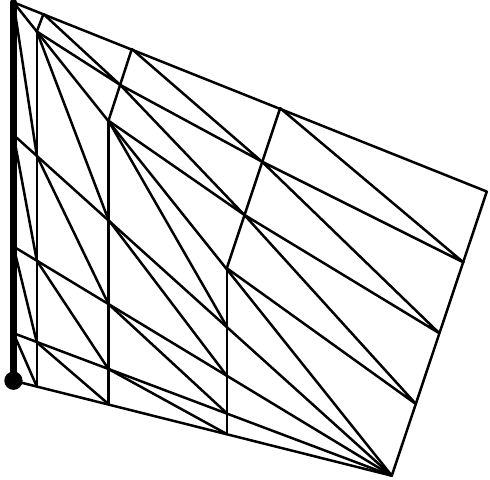}
    \caption{\label{fig:macros} Macroelements of types 1, 2, 3 and 4}
  \end{center}
\end{figure}

We point out that anisotropic elements can appear only in Cases 3 and
4, for which $\Th$ contains needle elements near the singular edge and
flat elements near the opposite one, see Figure \ref{fig:macros}.  We
further observe that if $\Lambda_\ell$ is of type 3 or 4, the elements
in $\Th|_{\Lambda_\ell}$ do not intersect any plane of $\mathcal
P_\ell$.

For each element $T$ we introduce its lengths
$h_{1,T},h_{2,T},h_{3,T}$ and $h_T$ as follows. Let $h_T$ be the
diameter of $T$. If $T\subset \Lambda_\ell$ with $\Lambda_\ell$ of
type 1 or 2, then $h_{1,T}=h_{2,T}=h_{3,T}=h_T$. If
$T\subset\Lambda_\ell$ with $\Lambda_\ell$ of type 3 or 4 then
$h_{3,T}$ is the length of the edge $e_{3,T}$ of $T$ parallel to the
singular edge, and $h_{1,T}=h_{2,T}=\frac12(|e_{1,T}|+|e_{2,T}|)$
where $e_{1,T}$ and $e_{2,T}$ are the edges of $T$ intersecting
$e_{3,T}$ and each one of them is contained in some plane of $\mathcal
P_\ell$.

By classical regularity theory, the solution $u$ of the boundary value
problem \eqref{eq:poissonproblem} is continuous, see e.g. \cite[page
page 79]{grisvard:92b}, such that the Lagrange interpolant $u_I$ with
respect to the subdivision $\{\Lambda_\ell\}$ is well defined. We
consider the decomposition
\begin{equation}\label{eq:splitting}
  u=u_I + u_R.
\end{equation}
It follows that the restriction $u_R|_{\Lambda_\ell}$ has the same
smoothness properties as~$u$, see Theorem
\ref{thm:regularity}. Furthermore, $u_R$ vanishes in coupling nodes
and on singular edges.  We construct now an interpolant $D_hu_R\in
V_h$ which also vanishes on these nodes such that $u_I+D_hu_R \in V_h$
can be used to estimate the discretization error via \eqref{eq:cea}.

To this end, let $\NN$, $\Nin$, $\Nc$ and $\Ns$ be the set of all
nodes of $\Th$, the set of all the interior nodes, the set of coupling
nodes, and the set of nodes which belong to some singular edge,
respectively. The terminal points of the singular edges are included
in $\Ns$. The piecewise linear nodal basis on $\Th$ is denoted by
$\{\phi_\node\}_{\node\in\mathcal{N}}$. We associate (as specified
below) with each $\node\in\mathcal N\setminus(\Nc\cup\Ns)$ an edge
$\sigma_\node$ with $\node$ as an endpoint. Note that
$u|_{\sigma_\node}\in L^2(\sigma_\node)$ since $u\in H^s(\Omega)$ with
$s>1$.  Hence the operator $D_h$ with
\begin{equation}\label{eq:defDh}
  D_hu=\sum_{\node\in\NN\setminus(\Nc\cup\Ns)}
  (\Pi_{\sigma_\node}u)(\node) \cdot \phi_\node(x),
\end{equation}
is well defined when $\Pi_{\sigma}: L^2(\sigma)\to
\mathcal{P}_1(\sigma)$ is the $L^2(\sigma)$-projection operator onto
the space of polynomials of degree less than or equal to one. Note
that $D_hu$ vanishes on coupling nodes and on singular edges by
construction. In order to impose the boundary conditions and to be
able to prove interpolation error estimates we need to select the
edges $\sigma_\node$ in an appropriate way, compare the illustration
in Figure \ref{fig:sigmas}.
\begin{figure}
  \begin{center}
    \includegraphics[width=40mm]{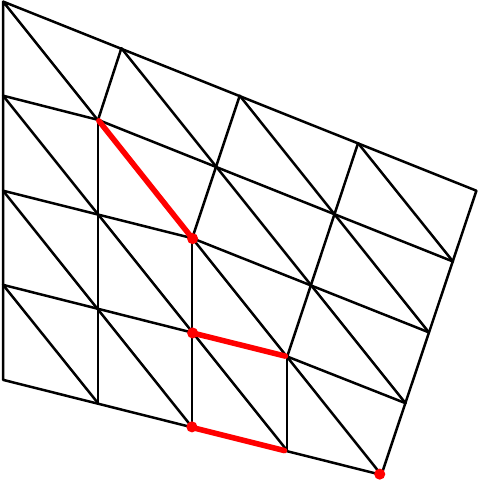} \qquad
    \includegraphics[width=40mm]{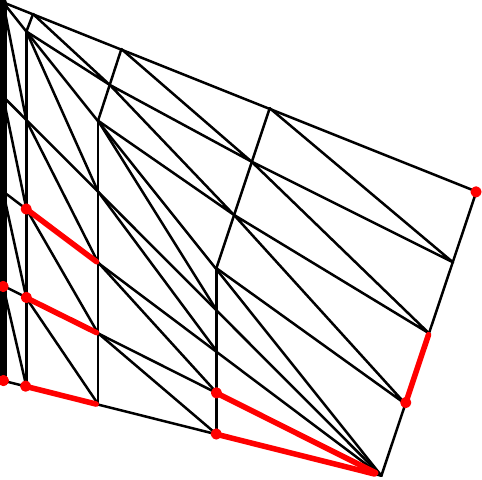}
  \end{center}
  \caption{\label{fig:sigmas}Illustration of the edges $\sigma_\node$}
\end{figure}
First, we demand that
\begin{itemize}
\item for each node $\node\in\NN\setminus(\Nc\cup\Ns)$, $\node$ and
  $\sigma_\node$ belong to the same macroelement. \pagebreak[3]

  This requires in particular the following restrictions.
  \begin{itemize}
  \item If $\node$ lays on a boundary or coupling face, then
    $\sigma_\node$ is contained in that face.
  \item If $\node$ lays on a coupling edge, then $\sigma_\node$ is
    contained in that coupling edge.
  \end{itemize}
\end{itemize}
Note that these requirements made the treatment of the coupling nodes
via the interpolation on the initial $u_I$ necessary. Note further
that this construction leads to a preservation of the homogeneous
Dirichlet boundary condition.

In order to prove the stability of $D_h$ in the anisotropic refinement
regions we also require:
\begin{itemize}\setlength{\itemsep}{0ex}
\item If $\node$ is a vertex of a tetrahedron contained in a
  macroelement $\Lambda_\ell$ of types 3 or 4, then $\sigma_\node$ is
  an edge contained on some plane of $\mathcal P_\ell$.
\item If $\node_1$ and $\node_2$ belong to a macroelement
  $\Lambda_\ell$ of types 3 or 4 and have the same orthogonal
  projection onto the $x_1^{(\ell)}x_2^{(\ell)}$-plane, then the same
  holds for $\sigma_{n_1}$ and ${\sigma_{n_2}}$.
\end{itemize}

In order to estimate the interpolation error we need to define for
each $T\in \Th$ a set $S_T$ which should satisfy the following
assumptions.
\begin{itemize}\setlength{\itemsep}{0ex}
\item The set $S_T$ is a union of elements of $\Th$ (plus some faces)
  and in particular $T\subseteq S_T$.
\item The set $S_T$ is an open connected domain, and as small as possible.
\item We have $\sigma_\node\subset \overline{S_T}$ for all nodes $\node$ of $T$.
\item If $T\subset \Lambda_\ell$, then $S_T\subset \Lambda_\ell$.
\item If $T\subset\Lambda_\ell$ with $\Lambda_\ell$ of type 3 of 4,
  then $S_T$ is a prism where the top and bottom faces are contained
  in two planes of $\mathcal P_\ell$ (and so they are not parallel)
  and the other faces are parallel to the singular edge.
\end{itemize}\pagebreak[3]

The following properties follow from the definitions of the edges
$\sigma_\node$ and the sets $S_T$.
\begin{enumerate}
\item Let $T$ be contained in a macroelement $\Lambda_\ell$ of type
  $3$ or $4$. If $\overline T$ intersects two planes $p_1$ and $p_2$ of
  $\mathcal P_\ell$, then $\overline{S_T}$ intersects exactly the same
  planes $p_1$ and $p_2$.
\item If the node $\node$, $\node\not\in\Nc\cup\Ns$, belongs to a
  coupling face, that means that there exist tetrahedra
  $T_1\subset\Lambda_{\ell_1}$ and $T_2\subset\Lambda_{\ell_2}$ with
  $\ell_1\ne\ell_2$ and $\node\in \overline{T_1}\cap\overline{T_2}$,
  then $S_{T_1}\cap S_{T_2}=\emptyset$ but $\sigma_\node\subset
  \overline{S_{T_1}}\cap \overline{S_{T_2}}$.
\item If $T$ is an isotropic element then all the elements in $S_T$
  are also isotropic and of size of the same order.
\end{enumerate}
The second point is essential for our proof of the approximation
properties. It was the target for which we made the construction as it is.


\setcounter{equation}{0}
\section{\label{sec:mainresults}Error estimates}

The aim of this section is to derive error estimates for our
discretization. They are based on local interpolation error estimates
for our interpolant~$D_h$. For proving these estimates we have to
distinguish several cases, see also Figure \ref{fig:cases} for an
illustration:
\begin{enumerate}
\setlength{\itemsep}{0ex}
\item $T$ is an isotropic element without coupling node, $u$ has full
  regularity,
\item $T$ is an isotropic element with coupling node, $u$ has full
  regularity,
\item $T$ is an isotropic element with coupling node, $u$ has reduced
  regularity,
\item $T$ is an anisotropic flat element without coupling node, $u$
  has full regularity,
\item $T$ is an anisotropic flat element with coupling node, $u$ has
  full regularity,
\item $T$ is an anisotropic needle element without node on the
  singular edge, $u$ has full regularity,
\item $T$ is an anisotropic needle element with node on the singular
  edge, $u$ has reduced regularity.
\end{enumerate}
\begin{figure}
  \begin{center}
    \includegraphics[width=40mm]{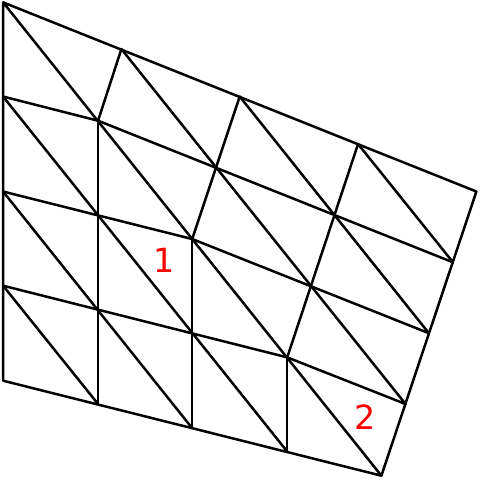} \qquad 
    \includegraphics[width=40mm]{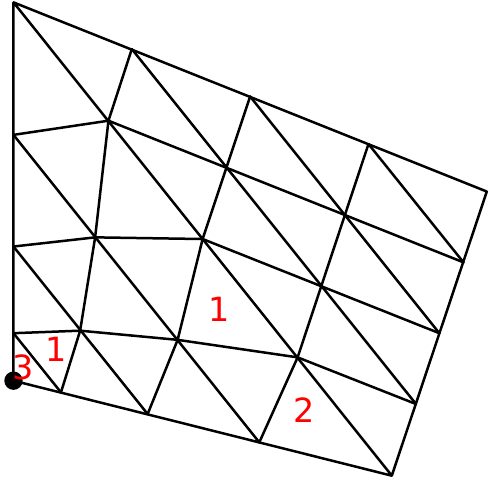} \\
    \includegraphics[width=40mm]{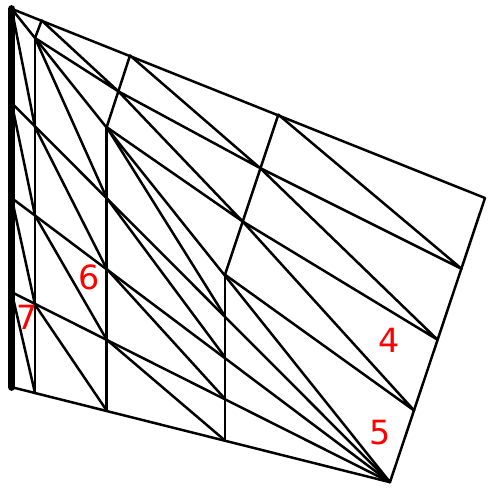} \qquad
    \includegraphics[width=40mm]{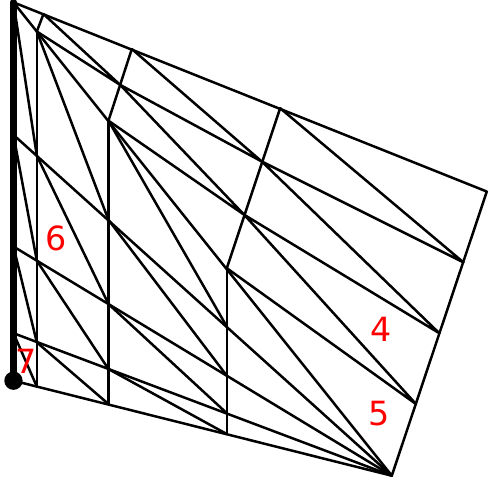}
  \end{center}
  \caption{\label{fig:cases} Illustration of the cases that have to
    considered for the interpolation error estimates}
\end{figure}

In Lemma~\ref{lem:iso_regular} we present the general approach for the
proof of the local interpolation error estimate by considering isotropic
elements with and without coupling nodes (cases 1 and 2). We proceed
with Lemmas~\ref{lem:iso_weighted} where we introduce for isotropic
elements how to cope with the weighted norms in the case of reduced
regularity (case 3). The interpolated function is only from a weighted
Sobolev space but we will see that this even simplifies some parts of
the proof.

For anisotropic elements the use of an inverse inequality (as was done
in the previous lemmas) has to be avoided; instead we use the
structure of the meshes in the macroelements of types 3 and 4. We
start with a stability estimate of $\partial_3 D_hu$ which allows
immediately the treatment of anisotropic flat elements (cases 4 and 5)
in Lemma \ref{lem:aniso_flat_regular}. Then we prove stability
estimates for the remaining derivatives and continue with the
interpolation error estimates for needle elements.
Lemma~\ref{lem:aniso_needle_regular} is devoted to case 6, and Lemma
\ref{lem:aniso_weighted} to case~7.

All these local estimates can then be combined to prove the global
interpolation error estimate, see Theorem~\ref{lem:globinterrest}, and
the finite element error estimates, see Corollary~\ref{cor:fe_errest}.

\begin{lemma}[isotropic element, full regularity]
  \label{lem:iso_regular}
  If $T$ is an isotropic element then the local interpolation error
  estimate
  \begin{align}
    \label{eq:interrest_isoH2} |u-D_hu|_{H^1(T)} &\le
    Ch_T|u|_{H^2(S_T)}
  \end{align}
  holds provided that $u\in H^2(S_T)$ and $u(\node)=0$ for all $n\in\Nc$.
\end{lemma}

\begin{proof} 
  Following the explanations in \cite[page 486]{scott:90} and
  \cite[page 1156]{apel:97b}, an explicit representation of $D_hu$
  from \eqref{eq:defDh} can be given by introducing the unique
  function $\psi_\node\in V_h|_{\sigma_\node}$ with
  $\int_{\sigma_\node}\psi_\node\phi_j=\delta_{\node j}$ for all
  $j\in\mathcal{N}$ such that
  \begin{equation}\label{def:PiSigmaUnode}
    (\Pi_{\sigma_\node}u)(\node)=\int_{\sigma_\node}u\psi_\node
  \end{equation}
  and
  \begin{equation}\label{eq:defDhnew}
    D_hu|_T=\sum_{\node\in \NT}
    \left(\int_{\sigma_{\node}}u\psi_\node\right)\cdot\phi_{\node}
  \end{equation}
  where we denote by $\NT$ the set of nodes of $T$ without the
  coupling nodes. Note that
  \begin{equation}\label{eq:L2sigma}
    \|\psi_\node\|_{L^\infty(\sigma_{\node})}=C\,|\sigma_{\node}|^{-1},
  \end{equation}
  compare \cite[page 1157]{apel:97b}. (By some calculation one can
  even specify that $C=4$.) With \eqref{eq:defDhnew}, the direct
  computation
  \begin{equation}\label{eq:normPhi}
    |\phi_{\node}|_{H^1(T)} \le C h_T^{-1} |T|^{1/2},
  \end{equation}
  the trace theorem
  \begin{equation}\label{eq:tracetheorem}
    \|u\|_{L^1(\sigma_{\node})} \le C |\sigma_{\node}| |S_T|^{-1/2}
    (\|u\|_{L^2(S_T)} + h_T |u|_{H^1(S_T)} + h_T^2 |u|_{H^2(S_T)}),
  \end{equation}
  and $|S_T|\le C|T|$   we obtain
  \begin{align}\nonumber
    |D_hu|_{H^1(T)} &\le C\sum_{\node\in \NT} \|u\|_{L^1(\sigma_{\node})}
    \|\psi_\node\|_{L^\infty(\sigma_{\node})} |\phi_{\node}|_{H^1(T)} \\
    &\le \label{eq:stab_iso}
    Ch_T^{-1}(\|u\|_{L^2(S_T)} + h_T |u|_{H^1(S_T)} + h_T^2 |u|_{H^2(S_T)}).
  \end{align}

  If $\NT$ does not contain a node $\node\in\Nc$ we find that $D_hw=w$
  for all $w\in\mathcal{P}_1$ such that we get by using the triangle
  inequality and the stability estimate \eqref{eq:stab_iso}
  \begin{align*}
    |u-D_hu|_{H^1(T)} &= |(u-w)-D_h(u-w)|_{H^1(T)}
    \quad\forall w\in\mathcal{P}_1 \\
    &\le |u-w|_{H^1(T)} + |D_h(u-w)|_{H^1(T)} \\
    &\le C \left(h_T^{-1}\|u-w\|_{L^2(S_T)} + |u-w|_{H^1(S_T)} + h_T |u|_{H^2(S_T)}\right).
  \end{align*}
  We use now a Deny--Lions type argument (see e.g. \cite{dupont:80})
  and conclude estimate~\eqref{eq:interrest_isoH2}.

  In the case when $\NT$ contains a node $\node\in\Nc$ we do not have
  the property that $D_hw=w$ for all $w\in\mathcal{P}_1$ but we can
  use that $u(\node)=0$.  Let $\sigma_n$ be an edge contained in $T$ having
  $\node$ as an endpoint, and let $\phi_\node$ be the Lagrange basis
  function associated with $\node$.
  (Note that we deal here with nodes $\node$ which are not used in the
  definition of $D_h$. Therefore we can assume that $\sigma_{\node}$
  is local in $\Lambda_\ell$.)
  Consequently, we have with the previous argument that
  \begin{equation}\label{eq:interp_coupling_node}
    \left|u-(D_hu+(\Pi_{\sigma_\node}u)(\node)\phi_\node)\right|_{H^1(T)} \le
    C h_T|u|_{H^2(S_T)}.
  \end{equation}
  Let $I_{T}u$ be the linear Lagrange interpolation of $u$ on $T$.
  Since $I_{T}u|_{\sigma_\node}$ is linear, we have have
  $(\Pi_{\sigma_\node}I_{T}u)(n)=0$. From this fact and using
  \eqref{def:PiSigmaUnode}--\eqref{eq:tracetheorem} as in the derivation of
  \eqref{eq:stab_iso} (here with the specific $T$ instead of $S_T$
  since $\sigma_\node\subset\overline T$), we have
  \begin{eqnarray*}
    \lefteqn{|(\Pi_{\sigma_\node}u)(\node)\phi_\node|_{H^1(T)}=
    |(\Pi_{\sigma_\node}(u-I_{T}u))(\node)\phi_\node|_{H^1(T)}}
    \\&\le&
    C h_T^{-1}\left( |u-I_{T}u|_{L^2(T)} + h_T|u-I_{T}u|_{H^1(T)} +
      h_T^2|u|_{H^2(T)} \right)\\
    &\le& C h_T|u|_{H^2(T)}
  \end{eqnarray*}
  where we used standard estimates for the Lagrange interpolant in the
  last step. With \eqref{eq:interp_coupling_node} and the triangle
  inequality we conclude estimate \eqref{eq:interrest_isoH2} also in
  this case.
\end{proof}

\begin{lemma}[isotropic element, reduced regularity]
  \label{lem:iso_weighted}
  If $T$ is an isotropic element then the local interpolation error
  estimate
  \begin{align}
    \label{eq:interrest_isoV} |u-D_hu|_{H^1(T)} &\le
    Ch_T^{1-\beta}\|u\|_{V^{2,2}_{\beta,0}(S_T)}
  \end{align}
  holds provided that $u\in V^{2,2}_{\beta,0}(S_T)$, $\beta\in[0,1)$.
\end{lemma}

\begin{proof}
  We start as in the proof of Lemma \ref{lem:iso_regular} but use
  the sharper trace theorem
  \[
    \|u\|_{L^1(\sigma_{\node})} \le C |\sigma_{\node}| |S_T|^{-1}
    (\|u\|_{L^1(S_T)} + h_T |u|_{W^{1,1}(S_T)} + h_T^2 |u|_{W^{2,1}(S_T)}).
  \]
  With \eqref{eq:defDhnew}, \eqref{eq:L2sigma}, \eqref{eq:normPhi},
  and $|S_T|\le C|T|$ we obtain
  \begin{align*}\nonumber
    |D_hu|_{H^1(T)} &\le C\sum_{\node\in \NT} \|u\|_{L^1(\sigma_{\node})}
    \|\psi_\node\|_{L^\infty(\sigma_{\node})} |\phi_{\node}|_{H^1(T)} \\
    &\le  C|S_T|^{-1/2}
    (h_T^{-1}\|u\|_{L^1(S_T)} + |u|_{W^{1,1}(S_T)} + h_T |u|_{W^{2,1}(S_T)}) \\
    &\le C (h_T^{-1}\|u\|_{L^2(S_T)} +
    |u|_{H^1(S_T)} + |S_T|^{-1/2} h_T |u|_{W^{2,1}(S_T)})
  \end{align*}
  and hence via the triangle inequality
  \begin{align}\label{eq:est_iso_W11}
    |u-D_hu|_{H^1(T)}
    &\le  C (h_T^{-1}\|u\|_{L^2(S_T)} +
    |u|_{H^1(S_T)} + |S_T|^{-1/2} h_T |u|_{W^{2,1}(S_T)}).
  \end{align}
  For the first two terms we just use that $R\le h_T$, hence $1\le
  h_TR^{-1}$, to get
  \begin{align*}
    \|u\|_{L^2(S_T)} &\le h_T^{2-\beta}\|u\|_{V^{0,2}_{\beta-2,0}(S_T)}, \\
    |u|_{H^1(S_T)} &\le h_T^{1-\beta}|u|_{V^{1,2}_{\beta-1,0}(S_T)}.
  \end{align*}
  To estimate the third term we use the Cauchy--Schwarz inequality and
  again $R\le h_T$, to obtain for $|\alpha|=2$
  \begin{align*}
    |D^\alpha u|_{L^1(S_T)} &\le
    \|R^{-\beta}\|_{L^2(S_T)} \|R^{\beta}D^\alpha u\|_{L^2(S_T)} \\ &\le
    C|S_T|^{1/2}h_T^{-\beta}|u|_{V^{2,2}_{\beta,0}(S_T)}
  \end{align*}
  where $\|R^{-\beta}\|_{L^2(S_T)}\le C|S_T|^{1/2}h_T^{-\beta}$ is
  obtained by executing the integration and using that
  $\beta<\frac32$.  All these estimates imply estimate
  \eqref{eq:interrest_isoV}.
\end{proof}

In order to prove interpolation error estimates for the anisotropic
elements we derive stability estimates for $D_h$ where we avoid
the use of the inverse inequality.  
Let $x_1,x_2$ and $x_3$ be a Cartesian coordinate system with the
$x_3$-direction parallel to the singular edge of $\Lambda$.  We will
estimate separately the $L^2$-norm of the derivatives of $D_hu$.

Let $T$ be an anisotropic element
with the characteristic lengths $h_{1,T}=h_{2,T}$ and $h_{3,T}$.
We will not use that $h_{3,T}\ge h_{j,T}$, $j=1,2$, in the next lemma
in order to use this estimate both for the needle and the flat
elements. 

\begin{lemma}[Stability in direction of the singular edge]
  \label{lem:stability_x3}
  For any anisotropic element $T$ the estimate
  \[
    \|\partial_3D_hu\|_{L^2(T)} \le C|S_T|^{-1/2} \sum_{|\alpha|\le1}
    h_T^\alpha \|D^\alpha \partial_3u\|_{L^1(S_T)}
  \]
  holds provided that $\partial_3u\in W^{1,1}(S_T)$.
\end{lemma}

\begin{proof}
  We observe that $T$ has an edge $e_T$ parallel to the singular edge,
  and so, parallel to the $x_3$-axis. Since $D_hu$ is linear on $T$,
  we have $\partial_3D_hu|_T=\partial_3D_hu|_{e_T}$. If $e_T$ is
  contained on the singular edge, then $\partial_3D_hu|_T=0$ since
  $D_hu|_{e_T}=u|_{e_T}=0$ and we are done. Now, consider the case
  that $e_T$ is not contained in a singular edge and denote its endpoints
  by $\node_1$ and $\node_2$ such that
  $\partial_3\phi_{\node_1}\big|_T=-h_{3,T}^{-1}$ and
  $\partial_3\phi_{\node_2}\big|_T=h_{3,T}^{-1}$. Then we have
  \[
    \partial_3D_hu = h_{3,T}^{-1} \left[
    \int_{\sigma_{\node_2}} u\psi_{\node_2} -
    \int_{\sigma_{\node_1}} u\psi_{\node_1} \right]
  \]
  We observe now that by our assumptions $\sigma_{\node_1}$ and
  $\sigma_{\node_2}$ have the same projection $\sigma_T$ into the
  $x_1x_2$-plane and hence form two opposite edges of a plane
  quadrilateral which is parallel to the $x_3$-axis and which we will
  denote by $F_T$. We note further that $\psi_{\node_1}$ and
  $\psi_{\node_2}$ can be considered as the same function $\psi_T$
  defined on $\sigma_T$ and
  $\|\psi_T\|_{L^\infty(\sigma_T)}=C|\sigma_T|^{-1}$.  With this
  insight we obtain
  \begin{align*}
    |\partial_3D_hu| &= h_3^{-1} \left|
    \int_{\sigma_{\node_2}} u\psi_{\node_2} -
    \int_{\sigma_{\node_1}} u\psi_{\node_1} \right| =
    h_3^{-1} \left|\int_{F_T} \partial_3u\psi_T \right| \\ &\le
    Ch_3^{-1} |\sigma_T|^{-1} \|\partial_3u\|_{L^1(F_T)} \le
    C |F_T|^{-1} \|\partial_3u\|_{L^1(F_T)}.
  \end{align*}
  We integrate this estimate over $T$, apply the standard trace theorem
  \[
    \|v\|_{L^1(F_T)} \le C|F_T||S_T|^{-1} \sum_{|\alpha|\le1}
    h_T^\alpha \|D^\alpha v\|_{L^1(S_T)}
  \]
  and obtain the desired estimate.
\end{proof}

We are now prepared to estimate the interpolation error for the flat
elements occurring far away from the singular edge in cases 3 and 4.

\begin{lemma}[anisotropic flat element, full regularity]
  \label{lem:aniso_flat_regular}
  If $T$ is an anisotropic flat element ($h_{3,T}\le h_{1,T}=
  h_{2,T}$) then the local interpolation error estimate
  \begin{align}
    \label{eq:interrest_anisoH2flat} |u-D_hu|_{H^1(T)} &\le
    Ch_T|u|_{H^2(S_T)}
  \end{align}
  holds provided that $u\in H^2(S_T)$. 
  (Remember that $h_T=\mathrm{diam}(T)$.)
\end{lemma}

\begin{proof}
  The proof for $\partial_3(u-D_hu)$ can be done on the basis of Lemma
  \ref{lem:stability_x3}. Assume for the moment that the element $T$
  does not contain a coupling node. Similar to the proof of Lemma
  \ref{lem:iso_regular} we obtain for any $w\in\mathcal P_1$
  \begin{align*}
    \|\partial_3(u-D_hu)\|_{L^2(T)} &=
    \|\partial_3(u-w)-\partial_3D_h(u-w)\|_{L^2(T)} \\ &\le
    C\|\partial_3(u-w)\|_{L^2(S_T)} + C\sum_{|\alpha|=1} h_T^\alpha
    \|D^\alpha\partial_3 u\|_{L^2(S_T)}.
  \end{align*}
  We choose now $w\in\mathcal P_1$ such that the constant
  $\partial_3w$ satisfies $\int_{S_T}\partial_3(u-w)=0$ and such that we can
  conclude by using the Poincar\'e--Friedrichs inequality (or again a
  Deny--Lions type argument)
  \[
    \|\partial_3(u-w)\|_{L^2(S_T)} \le C\sum_{|\alpha|=1} h_T^\alpha
    \|D^\alpha\partial_3 u\|_{L^2(S_T)}
  \]
  and hence
  \[
    \|\partial_3(u-D_hu)\|_{L^2(T)} \le C\sum_{|\alpha|=1} h_T^\alpha
    \|D^\alpha\partial_3 u\|_{L^2(S_T)} \le Ch_T|u|_{H^2(S_T)}.
  \]
  Note that the polynomial $w$ can be chosen such that it vanishes in
  three nodes of $T$. It is completely described by choosing the
  appropriate value at one endpoint of the edge of $T$ which is
  parallel to the $x_3$-axis. Since a possible coupling node is not an
  endpoint of this edge, the argument above can also be used in the
  case of coupling nodes.

  For the other directions we can proceed as in the proof of Lemma
  \ref{lem:iso_regular}.
  In the case of coupling nodes the interpolation error estimate
  $|u-I_{T}u|_{H^1(T)} \le C h_T|u|_{H^2(T)}$ is used there which does
  not hold for anisotropic elements. However, the estimate
  $\|\partial_i (u-I_{T}u)\|_{L^2(T)} \le C h_T|u|_{H^2(T)}$, $i=1,2$,
  does hold, see for example \cite{apel:92}.
\end{proof}

It remains to prove interpolation error estimates for needle elements
such that we will assume $h_{1,T}= h_{2,T}\le Ch_{3,T}$ for the next
lemmas.

\begin{lemma}[Stability in direction perpendicular to singular edge,
  an\-iso\-tropic needle element away from singular edge]
  \label{lem:stability_st1} Assume that the element $T$ does not
  contain a node $\node\in\Ns$ and that $h_{1,T}= h_{2,T}\le
  Ch_{3,T}$. Then for $i=1,2$ we have
  \begin{align} \label{4.37}
    \|\partial_i D_hu\|_{L^2(T)} &\le C  \left( |u|_{H^1(S_T)} +
      h_{3,T}|\partial_3 u|_{H^1(S_T)}\right)
  \end{align}
  provided that $u\in H^1(S_T)$ and $\partial_3u\in H^1(S_T)$.
\end{lemma}

\begin{proof}
  For each node $\node\in\NT$ we denote by $F_{\node,T}$ the top or
  bottom face of the prismatic domain $S_T$ such that $\node\in\overline
  F_{\node,T}$.  Observe that we have $\sigma_\node\subset\overline
  F_{\node,T} \subset \overline S_T$ for all $\node\in \NT$. Observe
  further that $F_{\node,T}$ is isotropic with diameter of order
  $h_{1,T}$ and recall the standard trace inequality
  \begin{equation}\label{trace:edge-face}
    \|v\|_{L^1(\sigma_\node)}\le C |\sigma_\node|
    |F_{\node,T}|^{-1} \left(\|v\|_{L^1(F_{\node,T})} +
      h_{1,T} |v|_{W^{1,1}(F_{\node,T})}\right)
  \end{equation}
  for all $v\in W^{1,1}(F_{\node,T})$. We need also the trace
  inequality
  \begin{equation}\label{trace:face-element}
    \|v\|_{L^1(F_{\node,T})}\le
    C|F_{\node,T}||S_T|^{-1}\left(\|v\|_{L^1(S_T)} +
      h_{3,T}\|\partial_3 v\|_{L^1(S_T)}\right)
  \end{equation}
  which can be proved by using Lemma \ref{lem:trace:prism} from page
  \pageref{lem:trace:prism} and the facts that $S_T$ is a union of
  prisms, and $F_{\node,T}$ is a face of $S_T$.

  Let $s_T$ be one of the short edges of~$T$ and denote its endpoints
  by $\node^1$ and $\node^2$. We use the same notation $s_T$ for the direction
  of this edge in order to denote by $\partial_{s_T}v=\nabla v\cdot
  s_T/|s_T|$ the directional derivative. In the following we first
  estimate $\|\partial_{s_T}D_hu\|_{L^2(T)}$. After that, the desired
  estimates \eqref{4.37} easily follow as we will show.

  Notice that if $\node\in\NT\setminus \{\node^1,\node^2\}$ we have
  $\partial_{s_T}\phi_\node=0$, and if $\node\in \{\node^1,\node^2\}$ then
  $\|\partial_{s_T}\phi_\node\|_{L^\infty(T)}= |s_T|^{-1}\le
  Ch_{1,T}^{-1}$. For all $w\in P_0(S_T)$ we have (and here we use
  that the element does not contain a node
  $\node\in\Nc\cup\Ns$)
  \begin{align}
    \|\partial_{s_T} D_hu\|_{L^2(T)} & = \|\partial_{s_T}
    D_h(u-w)\|_{L^2(T)}\nonumber \\ \nonumber &\le \sum_{\node\in
      \NT\cap s_T} \left|\int_{\sigma_{\node}}(u-w)\psi_{\node}\right|
    \|\partial_{s_T}\phi_{\node}\|_{L^2(T)}\\ \nonumber &\le C
    h_{1,T}^{-1}|T|^{1/2} \sum_{\node\in\NT\cap s_T}
    \|u-w\|_{L^1(\sigma_{\node})}\|\psi_{\node}\|_{L^\infty(\sigma_{\node})} \\
    &\le C h_{1,T}^{-1}|T|^{1/2} \sum_{\node\in\NT\cap s_T}
    |\sigma_\node|^{-1} \|u-w\|_{L^1(\sigma_\node)}. \label{eq:oben}
  \end{align}
  From the trace inequality \eqref{trace:edge-face} we have for each
  $\node\in\NT\cap s_T$
  \[
  \|u-w\|_{L^1(\sigma_\node)}\le C|\sigma_\node|
  |F_{\node,T}|^{-1}\left(\|u-w\|_{L^1(F_{\node,T})} +
    h_{1,T}|u|_{W^{1,1}(F_{\node,T})}\right).
  \]
  Since the definition of $F_{\node,T}$ implies
  $F_{\node^1}=F_{\node^2}=:F_T$, we have
  \begin{align*}
    \|u-w\|_{L^1(\sigma_\node)} \le C|\sigma_\node| |F_T|^{-1} \left(
      \|u-w\|_{L^1(F_T)} + h_{1,T} |u|_{W^{1,1}(F_T)} \right).
  \end{align*}
  Now we choose $w$ as the average of $u$ on $F_T$ and use
  a Poincar\'e type inequality on $F_T$ to get
  \begin{align*}
    \|u-w\|_{L^1(\sigma_\node)} \le C|\sigma_\node| |F_T|^{-1} h_{1,T}
    |u|_{W^{1,1}(F_T)}.
  \end{align*}
  Therefore we arrive at
  \begin{align}\nonumber 
    \|\partial_{s_T} D_hu\|_{L^2(T)} &\le C |T|^{1/2} |F_T|^{-1}
    |u|_{W^{1,1}(F_T)} \\\nonumber &\le C |T|^{1/2} |S_T|^{-1} \left(
      |u|_{W^{1,1}(S_T)} + h_{3,T}|\partial_3 u|_{W^{1,1}(S_T)}\right)
    \\ \label{eq:dstDhu_aniso} &\le C |S_T|^{-1/2} \left( |u|_{W^{1,1}(S_T)} +
      h_{3,T}|\partial_3 u|_{W^{1,1}(S_T)}\right)
    \\ \nonumber &\le C  \left( |u|_{H^1(S_T)} +
      h_{3,T}|\partial_3 u|_{H^1(S_T)}\right)
  \end{align}
  where we used again the trace inequality \eqref{trace:face-element}.
  \enlargethispage*{3ex}

  Now, let $s_{1,T}$ and $s_{2,T}$ be two different short edges (edge
  vectors) of $T$ such that the determinant of the matrix made up of
  $\frac{s_{1,T}}{|s_{1,T}|}$, $\frac{s_{2,T}}{|s_{2,T}|}$ and ${\bf
    e_3}$ as columns is greater than a constant depending only the
  maximum angle of $T$. Note that this is possible due to the maximal
  angle condition, see \cite{jamet:76}. Then, if the canonical vector
  $\mathrm{e}_i$, $i=1,2$, is expressed as
  \[
    \mathrm{e}_i = c_{1,i} \frac{s_{1,T}}{|s_{1,T}|} + c_{2,i}
    \frac{s_{2,T}}{|s_{2,T}|} + c_{3,i} \mathrm{e}_3,
  \]
  it follows that $c_{1,i}$, $c_{2,i}$ and $c_{3,i}$ are bounded by
  above by a constant depending only on the maximum angle condition.
  Since
  \[
    \partial_i = c_{1,i} \partial_{s_{1,T}} + c_{2,i}
    \partial_{s_{2,T}} + c_{3,i} \partial_3
  \]
  we obtain \eqref{4.37} from \eqref{eq:dstDhu_aniso} with
  $s_T=s_{1,T}$ and $s_T=s_{2,T}$, Lemma \ref{lem:stability_x3}, and
  recalling that $h_{1,T}= h_{2,T}\le C h_{3,T}$.
\end{proof}

\begin{lemma}[Stability in direction perpendicular to singular edge,
  an\-iso\-tropic needle element at the singular edge]
  \label{lem:stability_st2} Assume that the element $T$ contains at
  least one node  $n\in\Ns$ and that $h_{1,T}= h_{2,T}\le C
  h_{3,T}$.  Then we have for $i=1,2$
  \begin{align} \nonumber
    \lefteqn{\|\partial_i D_hu\|_{L^2(T)}}  \\ \label{eq:stab_aniso_part2}
    & \le C|S_T|^{-1/2}\left(|u|_{W^{1,1}(S_T)} +
      \frac{h_{3,T}}{h_{i,T}} \|\partial_3u\|_{L^1(S_T)} +
      \sum_{|\alpha|=1} h_T^\alpha |D^\alpha u|_{W^{1,1}(S_T)}
    \right)
  \end{align}
  provided that $u\in W^{2,1}(S_T)$.
\end{lemma}

\begin{proof}
  For each node $n\in \Ns$ of $T$ we select one short edge
  $\sigma_{n}$ with an endpoint at $n$ and contained in the same
  macroelement as $T$ such that we can apply Lemma
    \ref{lem:stability_st1}. We have for $i=1,2$
  \begin{align} \nonumber 
    \lefteqn{\left\|\partial_i\left(D_hu +
      \sum_{n\in \Ns\cap \overline T} (\Pi_{\sigma_{n}}u)(n)
      \,\phi_{n} \right)\right\|_{L^2(T)}} \\ \label{eq:aux3}
    & \le C |S_T|^{-1/2} \left( |u|_{W^{1,1}(S_T)} + 
      h_{3,T}|\partial_3 u|_{W^{1,1}(S_T)}\right).
  \end{align}
  Now we deal with $\|\partial_i[(\Pi_{\sigma_{n}}u)(n)
  \phi_{n}]\|_{L^2(T)}$ which is first estimated by
  \begin{equation}\label{eq:aux1}
    \|\partial_i[(\Pi_{\sigma_{n}}u)(n) \phi_{n}]\|_{L^2(T)} \le C 
    \|\partial_i\phi_{n}\|_{L^2(T)}|\sigma_{n}|^{-1}\|u\|_{L^1(\sigma_{n})}
  \end{equation}
  for each $n\in \Ns\cap \overline T$.

  Let $n\in \Ns\cap \overline T$ and be $F_{n,T}$ be the face of $S_T$
  having $\sigma_{n}$ as an edge and another edge on the singular
  edge. Let $P_{n,T}$ be the greatest parallelogram contained in
  $F_{n,T}$ and having $\sigma_{n}$ as an edge.  So, $P_{n, T}$ is
  parallel to the $x_3$-axis, and its area is comparable with the area
  of $F_{n,T}$ since opposite edges of the trapezoid $F_{n,T}$ have
  equivalent length. Using a trace inequality we have
  \[
    \|u\|_{L^1(\sigma_{n})} \le C|\sigma_{n}||F_{n,T}|^{-1} 
   (\|u\|_{L^1(P_{n,T})} + h_{3,T}\|\partial_3u\|_{L^1(P_{n,T})}).
  \]
  But, since $u=0$ on the edge of $P_{n,T}$ 
  contained on the singular edge we can use the Poincar\'e inequality
  to obtain
  \begin{equation}\label{ineq:traceedge}
    \|u\|_{L^1(\sigma_{n})} \le C|\sigma_{n}||F_{n,T}|^{-1} 
    (|\sigma_{n}|\|\partial_{\sigma_{n}}u\|_{L^1(P_{n,T})} + 
    h_{3,T}\|\partial_3u\|_{L^1(P_{n,T})}).
  \end{equation}
  From Lemma \ref{lem:trace_face} we have for all $v\in W^{1,1}(S_T)$
  \begin{align} \nonumber
    \lefteqn{\|v\|_{L^1(P_{n,T})}} \\ \label{ineq:traceface}
    & \le C |F_{n,T}||S_T|^{-1} \left(\|v\|_{L^1(S_T)} +
      |s_{1,T}|\|\partial_{s_{1,T}}v\|_{L^1(S_T)} +
      |s_{2,T}|\|\partial_{s_{2,T}}v\|_{L^1(S_T)}\right).
  \end{align}
  Using twice \eqref{ineq:traceface} we obtain from
  \eqref{ineq:traceedge}
  \begin{align}\nonumber
    \lefteqn{\|u\|_{L^1(\sigma_{n})}} \\ \nonumber &\le C
    |\sigma_{n}|^2|S_T|^{-1}
    \left(\|\partial_{\sigma_{n}}u\|_{L^1(S_T)} +
      |s_{1,T}|\|\partial_{s_{1,T}\sigma_{n}}u\|_{L^1(S_T)} +
      |s_{2,T}|\|\partial_{s_{2,T}\sigma_{n}}u\|_{L^1(S_T)}\right) +
    \\\label{eq:aux2} &\quad +C |\sigma_{n}||S_T|^{-1} h_{3,T}
    \left(\|\partial_3u\|_{L^1(S_T)} +
      |s_{1,T}|\|\partial_{s_{1,T}3}u\|_{L^1(S_T)} +
      |s_{2,T}|\|\partial_{s_{2,T}3}u\|_{L^1(S_T)}\right).
  \end{align}
  With the estimates
  \begin{eqnarray*}
    \|\partial_{\sigma_{n}}u\|_{L^1(S_T)} &\le &
    |u|_{W^{1,1}(S_T)},\\
    \|\partial_{s_{i,T}\sigma_{n}}u\|_{L^1(S_T)} &\le&
    |u|_{W^{2,1}(S_T)},\qquad i=1,2,\\
    \|\partial_{s_{i,T}3}u\|_{L^1(S_T)}&\le&
    |\partial_3u|_{W^{1,1}(S_T)},\qquad i=1,2,
  \end{eqnarray*}
  the inequality
  \[
  \|\partial_{i}\phi_{n}\|_{L^2(T)} \le C h_{i,T}^{-1}|T|^{1/2},
  \]
  and $|\sigma_{n}|\sim h_{i,T}$ ($i=1,2$) we obtain from
  \eqref{eq:aux1}
  \begin{align}\nonumber
    \|\partial_i[(\Pi_{\sigma_{n}}u)(n) 
      \phi_{n}]\|_{L^2(T)}  &\le C
    |S_T|^{-1/2}\left(|u|_{W^{1,1}(S_T)} +
    (h_{1,T}+h_{2,T})|u|_{W^{2,1}(S_T)}\right)  \\ \label{eq:aux4}
    &\quad + C |S_T|^{-1/2} 
    \frac{h_{3,T}}{h_{i,T}} \|\partial_3u\|_{L^1(S_T)}.
  \end{align}
  Finally, taking into account that, since $h_{1,T}= h_{2,T}\le
  Ch_{3,T}$, we have
  \[
  (h_{1,T}+h_{2,T})|u|_{W^{2,1}(S_T)} +
  h_{3,T}\|\partial_{3}u\|_{W^{1,1}(S_T)} \le C\sum_{|\alpha|=1}
  h_T^\alpha |D^\alpha u|_{W^{1,1}(S_T)},
  \]
  inequality \eqref{eq:stab_aniso_part2} follows from \eqref{eq:aux3}
  and \eqref{eq:aux4}.
\end{proof}

We are now prepared to estimate the interpolation error for needle elements.

\begin{lemma}[anisotropic needle element, full regularity]
  \label{lem:aniso_needle_regular}
  If $T$ is an anisotropic element with $h_{1,T}=h_{2,T}\le Ch_{3,T}$
  then the local interpolation error estimates
  \begin{align}
    \label{eq:interrest_anisoH2} |u-D_hu|_{H^1(T)} &\le
    C \sum_{|\alpha|=1} h_T^\alpha |D^\alpha u|_{H^1(S_T)}
  \end{align}
  hold provided that $u\in H^2(S_T)$.
\end{lemma}

\begin{remark}
  The estimate \eqref{eq:interrest_anisoH2} does not hold for the
  Lagrange interpolant, see \cite{apel:92}.
\end{remark}

\begin{proof} \textbf{(Lemma \ref{lem:aniso_needle_regular})}
  Since the needle elements with full regularity do not contain a
  coupling node we can apply both Lemmas \ref{lem:stability_x3} and
  \ref{lem:stability_st1}. That means we have shown that
  \begin{align*}
    |D_hu|_{H^1(T)} &\le C \sum_{|\alpha|\le1} 
     h_T^\alpha \|D^\alpha\partial_3u\|_{L^2(S_T)} +
    C \left( |u|_{H^1(S_T)} + h_{3,T}|\partial_3 u|_{H^1(S_T)} \right) \\ 
    &\le C\sum_{|\alpha|\le1} h_T^\alpha\, |D^\alpha u|_{H^1(S_T)}.
  \end{align*}

  We exploit now that $D_hw=w$ for all $w\in\mathcal{P}_1$.
  Consequently, we get
  \begin{align*}
    |u-D_hu|_{H^1(T)} &= |(u-w)-D_h(u-w)|_{H^1(T)}
    \quad\forall w\in\mathcal{P}_1 \\
    &\le |u-w|_{H^1(T)} + |D_h(u-w)|_{H^1(T)} \\
    &\le C\sum_{|\alpha|\le1} h_T^\alpha\, |D^\alpha(u-w)|_{H^1(S_T)}.
  \end{align*}
  We use now again a Deny--Lions type argument where the form of
  Lemma~1 in \cite{apel:97b} best suits our needs, and conclude the
  desired estimate.
\end{proof}

\begin{lemma}[anisotropic needle element, reduced regularity]
  \label{lem:aniso_weighted}
  Let $T$ be an an\-iso\-tropic element with $h_{1,T}= h_{2,T} \le C
  h_{3,T}$ and let $S_T$ have zero distance to the singular edge. Then
  the local interpolation error estimate
  \begin{align}
    \label{eq:interrest_aniso_edge} |u-D_hu|_{H^1(T)} &\le
    Ch_{1,T}^{1-\delta} \sum_{i=1}^2
    \|\partial_i u\|_{V^{1,2}_{\delta,\delta}(S_T)}
    + Ch_{3,T} \|\partial_3 u\|_{V^{1,2}_{0,0}(S_T)}
  \end{align}
  holds provided that $u$ has the regularity demanded by the right-hand
  sides of the estimates and $\delta\in[0,1)$.  If $T$ is an element
  with $h_{1,T}= h_{2,T} \le C h_{3,T}$ and $S_T$ has zero
  distance to both a singular vertex and a singular edge then the
  local interpolation error estimate
  \begin{align}
    \label{eq:interrest_aniso_edgecorner} |u-D_hu|_{H^1(T)} &\le
    Ch_{1,T}^{1-\beta-\delta}h_{3,T}^\delta
    \sum_{i=1}^2 \|\partial_i u\|_{V^{1,2}_{\beta,\delta}(S_T)}
    + Ch_{1,T}^{-\beta}h_{3,T} \|\partial_3 u\|_{V^{1,2}_{\beta,0}(S_T)}
  \end{align}
  hold provided that $u$ has the regularity demanded by the right-hand
  sides of the estimates and $\beta,\delta\in[0,1)$, $\beta+\delta<1$.
\end{lemma}

\begin{proof}
  As in the proof of Lemma \ref{lem:aniso_needle_regular} we
  distinguish between the derivatives $\partial_3 D_hu$ and the derivatives
  along directions perpendicular to the $x_3$-axis.  From Lemma
  \ref{lem:stability_x3} we obtain by using the triangle inequality
  and $|S_T|^{-1/2}\|\partial_3u\|_{L^1(S_T)}\le
    \|\partial_3u\|_{L^2(T)}$
  \begin{align*}
     \|\partial_3 (u-D_hu)\|_{L^2(T)}  &\le
     \|\partial_3u\|_{L^2(T)} + C|S_T|^{-1/2}
    \sum_{|\alpha|\le1}h^\alpha \|D^\alpha \partial_3u\|_{L^1(S_T)}
     \\ &\le  C  \|\partial_3u\|_{L^2(S_T)} + C|S_T|^{-1/2}
    \sum_{|\alpha|=1}h^\alpha \|D^\alpha \partial_3u\|_{L^1(S_T)}.
  \end{align*}
  For the estimate of $\partial_iD_hu$, $i=1,2$, we use Lemma
  \ref{lem:stability_st2}, from which we conclude that
  \begin{align*}
    \lefteqn{\|\partial_{i} (u-D_hu)\|_{L^2(T)}} \\ &\le C  |u|_{H^1(S_T)} +
  C|S_T|^{-1/2} \left(\frac{h_{3,T}}{h_{i,T}}\|\partial_3 u\|_{L^1(S_T)} + 
  \sum_{|\alpha|=1} h_T^\alpha|D^\alpha u|_{W^{1,1}(S_T)}\right).
  \end{align*}
  These two estimates can be summarized by using $h_{1,T}\le Ch_{3,T}$
  to
  \begin{align}\nonumber
    \lefteqn{|u-D_hu|_{H^1(T)}} \\ \label{eq:stab2}  &\le C
    |u|_{H^1(S_T)} +
    C |S_T|^{-1/2} \left(\frac{h_{3,T}}{h_{1,T}}\|\partial_3 u\|_{L^1(S_T)} +
    \sum_{|\alpha|=1} h_T^\alpha|D^\alpha u|_{W^{1,1}(S_T)}\right).
  \end{align}
  It remains to estimate the terms against the weighted norms.
  Firstly, we have
  \begin{align*}
    |u|_{H^1(S_T)} &\le
     \sum_{i=1}^2 \|R^{1-\beta}\theta^{1-\delta}\cdot
     R^{\beta-1}\theta^{\delta-1}\partial_i u\|_{L^2(T)} +
     \|R^{1-\beta}\theta\cdot
     R^{\beta-1}\theta^{-1}\partial_3 u\|_{L^2(T)} \\ &\le
    \sum_{i=1}^2     \max_{S_T} R^{1-\beta}\theta^{1-\delta}\,
    \|\partial_i u\|_{V^{1,2}_{\beta,\delta}(S_T)} +
    \max_{S_T} R^{1-\beta}\,
    \|\partial_3 u\|_{V^{1,2}_{\beta,0}(S_T)}.
  \end{align*}
  With $R^{1-\beta}\theta^{1-\delta}= r^{1-\delta}
  R^{-\beta}R^\delta\le r^{1-\beta-\delta} R^\delta \le C
  h_{1,T}^{1-\beta-\delta} h_{3,T}^\delta$ (where we used the
  assumption $\beta+\delta\le1$) and $R^{1-\beta}\le h_{3,T}^{1-\beta}
  \le h_{1,T}^{-\beta} h_{3,T}$ we derive
  \begin{align*}
    |u|_{H^1(S_T)} &\le  Ch_{1,T}^{1-\beta-\delta} h_{3,T}^{\delta}
    \sum_{i=1}^2 \|\partial_i u\|_{V^{1,2}_{\beta,\delta}(S_T)}
    + Ch_{1,T}^{-\beta} h_{3,T} \|\partial_3 u\|_{V^{1,2}_{\beta,0}(S_T)}.
  \end{align*}
  With $R^{1-\delta}\theta^{1-\delta}= r^{1-\delta} \le r^{1-\delta}
  \le C h_{1,T}^{1-\delta}$ (using that the exponent is positive) we
  derive also
  \begin{align*}
    |u|_{H^1(S_T)} &\le  Ch_{1,T}^{1-\delta}
    \sum_{i=1}^2 \|\partial_i u\|_{V^{1,2}_{\delta,\delta}(S_T)}
    + C h_{3,T} \|\partial_3 u\|_{V^{1,2}_{0,0}(S_T)}.
  \end{align*}

  Secondly, for $T$ intersecting the singular edge, but no singular
  vertices, we have
  \begin{align*}
    \frac{h_{3,T}}{h_{1,T}} \|\partial_3u\|_{L^1(S_T)}&\le
    \frac{h_{3,T}}{h_{1,T}}\|\partial_3u\|_{V^{1,2}_{0,0}(S_T)}
    \|r\|_{L^2(S_T)} 
    \le h_{3,T}|S_{T}|^{1/2} \|\partial_3u\|_{V^{1,2}_{0,0}(S_T)}.
  \end{align*}
  If $T$ has also a singular vertex, then we have with
  $R^{\beta-1}\theta^{-1}=R^\beta r^{-1}$
  \begin{align*}
    \frac{h_{3,T}}{h_{1,T}} \|\partial_3u\|_{L^1(S_T)}&\le
    \frac{h_{3,T}}{h_{1,T}} \|\partial_3u\|_{V^{1,2}_{\beta,0}(S_T)}
    \|R^{-\beta}r\|_{L^2(S_T)} 
    \le h_{3,T}h_{1,T}^{-\beta}|S_T|^{1/2}
    \|\partial_3u\|_{V^{1,2}_{\beta,0}(S_T)}
  \end{align*}
  where we used that 
  \begin{equation}\label{eq:aux5}
    \|R^{-\beta}r\|_{L^2(S_T)} \le
    \|r^{1-\beta}\|_{L^2(S_T)} \le Ch_{1,T}^{1-\beta}|S_T|^{1/2}
  \end{equation}
  which can be obtained by integration.
  The second derivatives in estimate
  \eqref{eq:stab2} are treated in a similar way.  For $i=1,2,3$ we get
  \begin{align*}
     \|\partial_{i3} u\|_{L^1(S_T)}  & \le
    \|R^{-\beta}\|_{L^2(S_T)} \|R^\beta\partial_{i3} u\|_{L^2(S_T)}  \le
    h_{1,T}^{-\beta}|S_T|^{1/2} \|\partial_3
    u\|_{V^{1,2}_{\beta,0}(S_T)}.
  \end{align*}
  For $i,j=1,2$ and supposing that $T$ does not have singular vertices we
  have
  \begin{eqnarray*}
    h_{i,T}\|\partial_{ij}u\|_{L^1(S_T)}&\le&
    h_{1,T}\|R^{-\delta}\theta^{-\delta}\|_{L^2(S_T)}
    \|R^\delta\theta^\delta\partial_{ij}u\|_{L^2(S_T)}\\ &\le&
    h_{1,T}^{1-\delta}|S_T|^{1/2}
    |\partial_iu|_{V^{1,2}_{\delta,\delta}(S_T)},
  \end{eqnarray*}
  where we used again an argument as in \eqref{eq:aux5}. If $T$ has a
  singular vertex, then
  \begin{eqnarray*}
    h_{i,T}\|\partial_{ij}u\|_{L^1(S_T)}&\le&
    h_{1,T}\|R^{-\beta}\theta^{-\delta}\|_{L^2(S_T)}
    \|R^\beta\theta^\delta\partial_{ij}u\|_{L^2(S_T)}.
  \end{eqnarray*}
  But, $R^{-\beta}\theta^{-\delta}=R^{-\beta+\delta}r^{-\delta}\le
  R^\delta r^{-\beta-\delta}\le h_{3,T}^\delta r^{-\beta-\delta}$, and
  so, since $\beta+\delta<1$, a similar argument as in \eqref{eq:aux5}
  give us
  \[
  \|R^{-\beta}\theta^{-\delta}\|_{L^2(S_T)}\le
  h_{3,T}^\delta \|r^{-\beta-\delta}\|_{L^2(S_T)}\le
  Ch_{1,T}^{-\beta-\delta} h_{3,T}^\delta |S_T|^{1/2}.
  \]
  Hence we have
  \[
  h_{i,T}\|\partial_{ij}u\|_{L^1(S_T)}\le Ch_{1,T}^{1-\beta-\delta}
  h_{3,T}^\delta |S_T|^{1/2}
  \|\partial_iu\|_{V^{1,2}_{\beta,\delta}(S_T)}.
  \]
  Therefore, the desired estimates are proved.
\end{proof}

\begin{theorem}[\textbf{global interpolation error estimate}] \label{lem:globinterrest}
  Let $u$ be the solution of the boundary value problem
  \eqref{eq:poissonproblem} with $f\in L^2(\Omega)$, and let $u_I,u_R$
  be the functions obtained from the splitting \eqref{eq:splitting}.
  Assume that the refinement parameters $\mu_\ell$ and $\nu_\ell$
  satisfy the conditions
  \begin{align}
    \label{cond:mu} \mu_\ell &< \lambda_{\mathrm{e}}^{(\ell)}, \\
    \label{cond:nu} \nu_\ell &< \lambda_{\mathrm{v}}^{(\ell)} + \frac12, \\
    \label{cond:munu} \frac1{\nu_\ell} + \frac1{\mu_\ell}
    \left(\lambda_{\mathrm{v}}^{(\ell)} - \frac12\right) &> 1,
  \end{align}
  $\ell=1,\ldots,L$.  Then the global interpolation error estimate
  \begin{equation}\label{eq:globinterrest}
    |u_R-D_hu_R|_{H^1(\Lambda_\ell)} \le Ch\|f\|_{L^2(\Lambda_\ell)}
  \end{equation}
  is satisfied.
\end{theorem}

\pagebreak[3]

\begin{proof}
  The proof can be carried out following the lines of the proof of
  Theorem 5.1 in \cite{apel:97c} with the setting $p=2$. Note that
  only a finite number (independent of $h$) of the $S_T$ overlap at
  any point.
\end{proof}

\begin{remark}
  The refinement conditions \eqref{cond:mu}--\eqref{cond:munu} were
  discussed in \cite{apel:97c} already: The conditions \eqref{cond:mu}
  and \eqref{cond:nu} balance the edge and vertex singularities. The
  third condition, \eqref{cond:munu}, follows from \eqref{cond:nu} in
  the case $\mu_\ell=\nu_\ell$; only in the case $\mu_\ell<\nu_\ell$
  it imposes a condition between $\mu_\ell$ and $\nu_\ell$ limiting
  the anisotropy of the mesh. For the Fichera example treated in
  Section \ref{sec:tests} we have
  $\lambda_{\mathrm{v}}^{(\ell)}\approx0.454$ and
  $\lambda_{\mathrm{e}}^{(\ell)}=\frac23$. With the choice
  $\nu_\ell=0.9$ the conditions \eqref{cond:mu} and \eqref{cond:munu}
  imply the choice $0.414<\mu_\ell<\frac23$. For $\nu_\ell=0.8$ we
  would get the weaker condition $0.184<\mu_\ell<\frac23$.

  Note also that in the absence of singularities we have set
  $\lambda_{\mathrm{e}}^{(\ell)}=\infty$ and/or
  $\lambda_{\mathrm{v}}^{(\ell)}=\infty$. In these cases we can set
  $\mu_\ell=1$ and/or $\nu_\ell=1$.
\end{remark}

\begin{corollary}[\textbf{$H^1$ and $L^2$ finite element error estimate}]
  \label{cor:fe_errest}
  Let $u$ be the solution of the boundary value problem
  \eqref{eq:poissonproblem}, and let $u_h$ be the corresponding finite
  element solution on a finite element mesh as constructed in Section
  \ref{sec:regularity} with grading parameters satisfying the conditions
  \eqref{cond:mu}--\eqref{cond:munu}. Then the discretization error
  can be estimated by
  \begin{align}
    \label{est:globH1} \|u-u_h\|_{H^1(\Omega)} &\le Ch \|f\|_{L^2(\Omega)}, \\
    \label{est:globL2} \|u-u_h\|_{L^2(\Omega)} &\le Ch^2 \|f\|_{L^2(\Omega)}.
  \end{align}
\end{corollary}

\begin{proof}
  We choose $v_h=u_I+D_hu_R$ in estimate \eqref{eq:cea} and observe
  that $u-v_h=u_R-D_hu_R$. With Lemma \ref{lem:globinterrest} we
  obtain the estimate \eqref{est:globH1}. The $L^2$-error estimate can
  be derived by the standard Aubin--Nitsche method.
\end{proof}

\begin{remark}
  A trivial conclusion from \eqref{est:globH1} is the stability estimate
  \begin{align}
    \label{stab:H1} \|u_h\|_{H^1(\Omega)} &\le C \|f\|_{L^2(\Omega)}
  \end{align}
  which we will need in Section \ref{sec:application1}.
\end{remark}

\begin{remark}\label{rem:type5}
  In macroelements of type 4 with $\mu_\ell=\nu_\ell<1$, Apel and
  Nicaise suggested in \cite{apel:97c} the use of a more elegant
  refinement strategy as depicted in Figure \ref{fig:macro5}. 
  \begin{figure}
    \begin{center}
      \includegraphics[width=40mm]{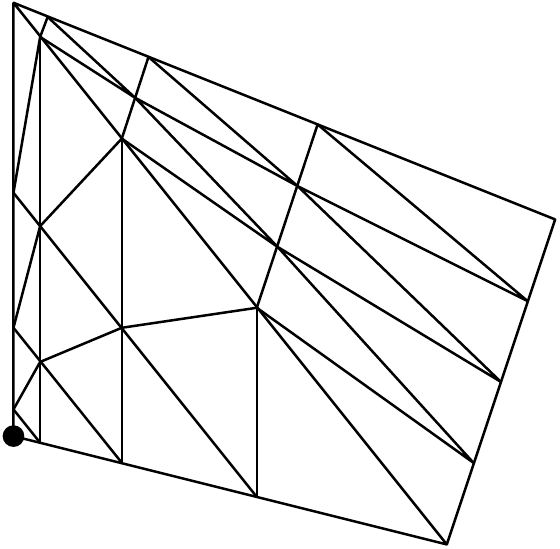}
    \end{center}
    \caption{\label{fig:macro5} Modification of macroelement of type
      4}
  \end{figure}
  Our proof cannot be transfered to this kind of mesh immediately
  since there may be elements $T$ where $S_T$ is not prismatic as it
  was exploited in the proof of Lemmas \ref{lem:stability_st1} and
  \ref{lem:stability_st2}. We conjecture that the assertion still
  holds but do not pursue this further in this paper.
\end{remark}

\setcounter{equation}{0}
\section{\label{sec:tests}Numerical test}

As in \cite{apel:97c} we consider the Poisson problem
\eqref{eq:poissonproblem} in the ``Fichera domain'' $\Omega:=(-1,1)^3
\setminus [0,1]^3$ and choose the right-hand side
$f=1+R^{-3/2}\ln^{-1}(R/4)$ which is in $L^2(\Omega)$ but not in
$L^p(\Omega)$ for $p>2$. For this problem we have
$\lambda_v\approx0.45$ for the concave vertex \cite{schmitz:93} and
$\lambda_e=\frac{\pi}{\omega_0}=\frac{2}{3}$ for the three concave
edges. All other edges and vertices are non-singular.

This boundary value problem was solved on quasi-uniform  and on
graded meshes with our refinement strategy using
$\mu=\nu=0.5<\min\{\lambda_e,\lambda_v+\frac12\}$, where types 1, 2
and 4 occur. Additionally we include the strategy where the macros of
type 4 are replaced by type 5, compare Remark \ref{rem:type5}.
Pictures of such meshes can be found in \cite{apel:97c}. The
refinement strategies and an a posteriori error estimator of residual
type \cite{siebert:93} were implemented into the finite element
package MooNMD \cite{john:04}. The estimated error norms are plotted
against the number of unknowns in Figure \ref{fig:diagram}.
\begin{figure}[tb]
  \begin{center}
    \includegraphics[width=0.85\textwidth]{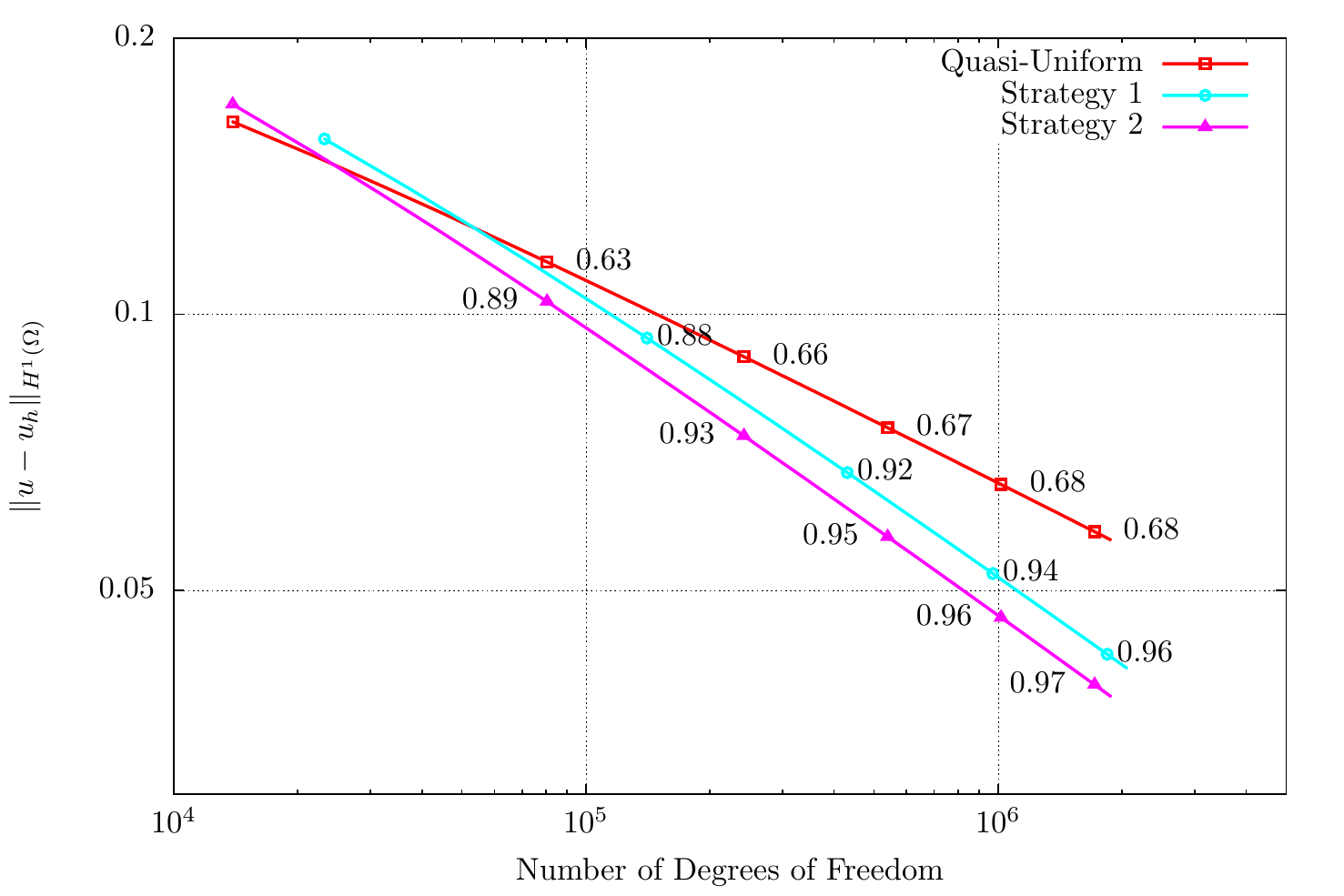}
  \end{center}
  \caption{\label{fig:diagram}Plot of the estimated error against the
    number of unknowns. The labels at the curve denote the estimated
    convergence order in terms of $h\sim N^{-1/3}$.}
\end{figure}
We see that the theoretical approximation order $h^1\sim N^{-1/3}$
from Corollary \ref{cor:fe_errest} can be verified in the practical
calculation for both refinement strategies.  The error with the second
strategy is slightly smaller. We denoted by $N$ the
number of nodes.

\pagebreak[4]

\setcounter{equation}{0}
\section{\label{sec:application1}Discretization error estimates for a
  distributed optimal control problem}

\newcommand{\Uad}{U^{\mathrm{ad}}}

Hinze introduced the variational discretization concept for
linear-quadratic control constrained optimal control problems in
\cite{hinze:05a}. We follow here this concept in a special case.
Consider the the optimal control problem
\begin{align*}
  \min_{(y,u)\in H^1_0(\Omega)\times\Uad} J(y,u) :=
  \frac{1}{2}\|y-y_d\|^2_{L^2(\Omega)} + \frac{\alpha}{2}\|u\|_{L^2(\Omega)}^2,
\end{align*}
where the state $y\in H^1_0(\Omega)$ is the weak solution of the
Poisson problem
\begin{equation}\label{eq:poissonproblem_ocp}
  -\Delta y=u \quad\mbox{in }\Omega, \qquad
  y=0 \quad\mbox{on }\partial\Omega,
\end{equation}
and the control $u$ is constrained by constant bounds $a_a,u_b\in\R$,
this means that the set of admissible controls is defined by
\[
\Uad:=\{u\in L^2(\Omega): u_a\le u\le u_b \text{ a.e. }\Omega\}.
\]
The regularization parameter $\alpha$ is a fixed positive number and
$y_d\in L^2(\Omega)$ is the desired state. It is well known that this
problem has a unique optimal solution $(\bar y, \bar u)$. There is an
optimal adjoint state $\bar p\in H^1_0(\Omega)$, and the triplet
$(\bar y, \bar u,\bar p)$ satisfies the first order optimality
conditions
\begin{alignat*}{2}
  (\nabla\bar y,\nabla v)_{L^2(\Omega)} &= (\bar u,v)_{L^2(\Omega)}
  \quad&&\forall v\in H^1_0(\Omega), \\
  (\nabla\bar p,\nabla v)_{L^2(\Omega)} &= (\bar y-y_d,v)_{L^2(\Omega)}
  \quad&&\forall v\in H^1_0(\Omega), \\
  (\alpha\bar u+\bar p,u-\bar u)_{L^2(\Omega)} &\ge 0
  \quad&&\forall u\in\Uad.
\end{alignat*}
With the variational discretization concept the approximate solution is
obtained by replacing $H^1_0(\Omega)$ by a finite element space
$V_h\subset H^1_0(\Omega)$ and searching $(\bar y_h, \bar u_h,\bar
p_h)\in V_h\times \Uad\times V_h$ such that
\begin{alignat*}{2}
  (\nabla\bar y_h,\nabla v_h)_{L^2(\Omega)} &= (\bar u_h,v_h)_{L^2(\Omega)}
  \quad&&\forall v_h\in V_h, \\
  (\nabla\bar p_h,\nabla v_h)_{L^2(\Omega)} &= (\bar y_h-y_d,v_h)_{L^2(\Omega)}
  \quad&&\forall v_h\in V_h, \\
  (\alpha\bar u_h+\bar p_h,u-\bar u_h)_{L^2(\Omega)} &\ge 0
  \quad&&\forall u\in\Uad.
\end{alignat*}
Note that the control space is not discretized; nevertheless $\bar
u_h$ can be obtained by the projection of $-\bar p_h/\alpha$ onto $\Uad$,
see \cite{hinze:05a}. The discretization error estimate
\begin{align*}
  \|\bar u - \bar u_h \|_{L^2(\Omega)} \!+\! 
  \|\bar y - \bar y_h \|_{L^2(\Omega)} \!+\! 
  \|\bar p - \bar p_h \|_{L^2(\Omega)} &\leq Ch^2
  \left(\|\bar u\|_{L^2(\Omega)} \!+\! \|y_d\|_{L^2(\Omega)} \right)
\end{align*}
can be concluded from \eqref{est:globL2} and \eqref{stab:H1}, see
\cite{hinze:05a,apel:10a}. With the proof of Corollary
\ref{cor:fe_errest} we have established this result for anisotropic
discretizations of the state equation \eqref{eq:poissonproblem_ocp} in
the case of three-dimensional polyhedral domains.

\section{\label{sec:application2}Discrete compactness property for
  edge elements}

The Discrete Compactness Property  is a
useful tool to study the convergence of finite element
discretizations of the Maxwell equations, both for eigenvalue and
source problems. It was first introduced by Kikuchi
\cite{Kikuchi:89} and proved for N\'ed\'elec edge elements of
lowest order on tetrahedral shape regular meshes. We refer to the
monograph by Monk \cite{Monk:03} and the references therein for
further analysis on isotropic meshes. The property was also
analyzed on anisotropically refined tetrahedral meshes on
polyhedra for edge elements of lowest order by Nicaise
\cite{Nicaise:01} (excluding corner singularities) and by Buffa,
Costabel, and Dauge \cite{Buffa_Costabel_Dauge:05}.

Lombardi \cite{Lombardi:13} extended this result to edge elements
of arbitrary order, also including corners and edge singularities.
The proof is based on two tools: 1) interpolation error estimates
for edge elements on meshes satisfying the maximum angle
condition, and 2) interpolation error estimates for a piecewise
linear interpolation operator defined on $W^{2,p}(\Omega)\cap
H^1_0(\Omega)$, $p\ge2$, preserving boundary conditions. For the
latter, the Lagrange interpolation was used (implying $p>2$)
together the results of Apel and Nicaise \cite{apel:97c}, giving
some artificial restrictions on the grading parameters defining
the allowed anisotropically graded meshes. Using now estimate
\eqref{est:globH1} of Corollary \ref{cor:fe_errest} we can extend
the result of \cite{Lombardi:13} allowing little more general
meshes.

In what follows we define a family of edge element spaces and
introduce the DCP for this family. We refer to \cite{Lombardi:13}
for further definitions and notation.
First we introduce the divergence-free space
\[
X=\left\{{\bf v}\in H_0({\bf curl},\Omega): \mbox{div\,} {\bf v}=0\mbox{ on }
\Omega\right\}.
\]
Then we introduce discretizations of this space where the
divergence-free condition is weakly imposed. Let $\mathrm I$ be a
denumerable set of positive real numbers having $0$ as the only limit
point. From now till the end of this section, we assume that $h\in
\mathrm I$. For each $h$, let $\mathcal{T}_h$ be the mesh on the
polyhedron $\Omega$ constructed in Section \ref{sec:regularity}.
Given an integer $k\ge1$, let $X_h$ be the space defined as
\[
X_h=\left\{{\bf v}_h\in H_0({\bf curl},\Omega): {\bf v}_h|_T\in
\mathcal N_k(T)\,\forall T\in \mathcal T_h, (\nabla p_h,{\bf
v}_h)=0\,\forall p_h\in S_h\right\}
\]
where $\mathcal N_k(T)$ is the space of edge elements of order $k$
on $T$, and
\[
S_h=\left\{p_h\in H^1_0(\Omega): p_h|_T\in \mathcal P_k(T)\,
\forall T\in\mathcal T_h\right\}.
\]
We say that the family of spaces $\{X_h\}_{h\in\mathrm I}$
satisfies the discrete compactness property if for each sequence
$\{{\bf v}_h\}_{h\in\mathrm J}$, $\mathrm J\subseteq \mathrm I$,
verifying for a constant $C$
\begin{align*}
& {\bf v}_h\in X_h,\qquad \forall h\in\mathrm J,\\
& \|{\bf v}_{h}\|_{H_0({\bf curl},\Omega)}\le C, \quad \forall h\in\mathrm
J,
\end{align*}
there exists a function ${\bf v}\in X$ and a subsequence
$\{{\bf v}_{h_n}\}_{n\in\mathbb N}$ such that (for $n\to\infty$)
\begin{align*}
{\bf v}_{h_{n}}&\to {\bf v}\quad\mbox{in }L^2(\Omega)
\\ {\bf v}_{h_n}&\rightharpoonup {\bf v} \quad \mbox{weakly in }H_0({\bf curl},\Omega).
\end{align*}

\begin{theorem}
If the grading parameters defining the meshes $\mathcal T_h$
satisfy the conditions \eqref{cond:mu}--\eqref{cond:munu}, then
the family of spaces $\{X_h\}_{h>0}$ verifies the discrete
compactness property.
\end{theorem}

\begin{proof}
   Follow exactly the arguments used to prove Theorem 5.2
   of~\cite{Lombardi:13} taking into account that the inequality (4.21)
   of that paper is now a consequence of  estimate
   \eqref{est:globH1}.
\end{proof}

\begin{appendix}
\setcounter{equation}{0}
\section{Proof of trace inequalities}

\begin{lemma} \label{lem:trace:prism} Let $P$ be a triangular prism
  with vertices $v_i$, $i=1,\ldots,6$, where the face $v_1v_2v_3$ is
  opposite to the face $v_4v_5v_6$, and where the edges $v_1v_4$,
  $v_2v_5$, and $v_3v_6$ are parallel to the $x_3$-axis,
  see Figure \ref{fig:prism}. 
  \begin{figure}
    \begin{center}
      \includegraphics[height=4.5cm]{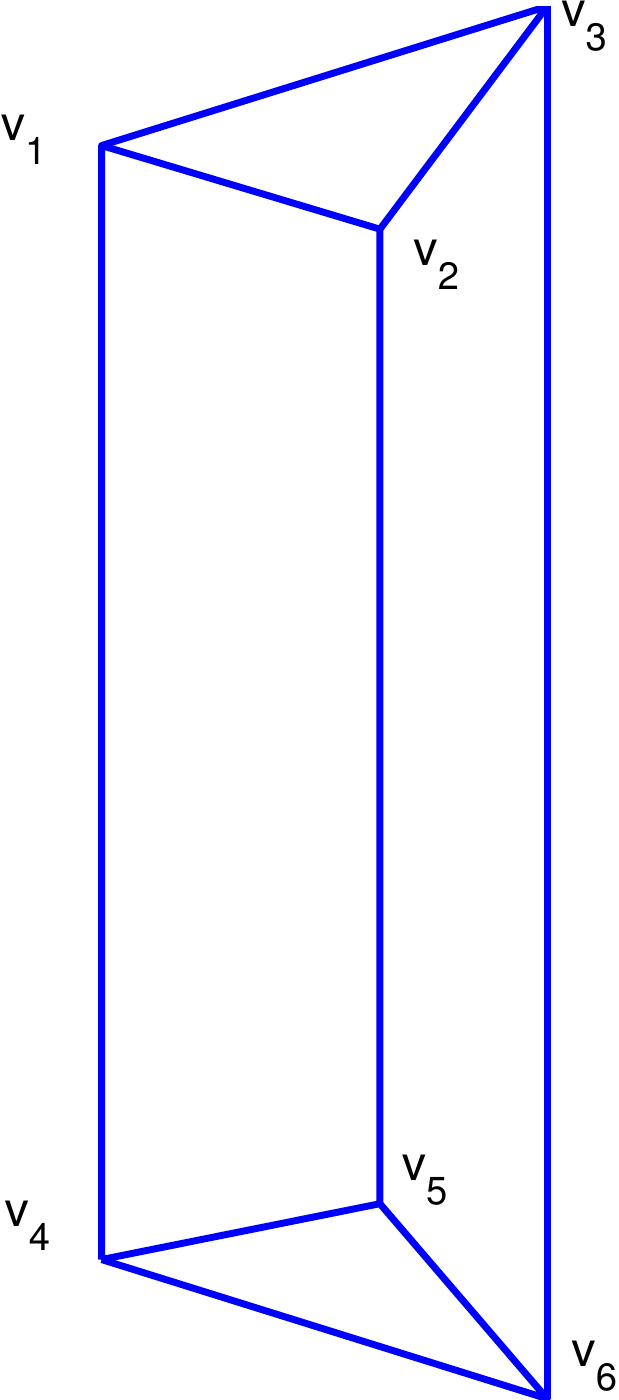}
    \end{center}
    \caption{\label{fig:prism}Illustration of the prism}
  \end{figure}
  Denote by $F$ the face $v_1v_2v_3$.
  Then for all $v\in W^{1,p}(P)$, $p\in[1,\infty)$, we have
  \[
    \|v\|_{L^p(F)}^p\le \frac{C_{reg}}{\cos\gamma} \cdot h_3^{-1}
    \left(\|v\|_{L^p(P)}^p + h_3^p\|\partial_3v\|_{L^p(P)}^p\right),
  \]
  where $h_3$ is length of the shortest vertical edge, and $\gamma$ is
  the angle between the $x_1x_2$-plane and the plane containing the
  face $F$. The constant $C_{reg}$ depends only on the minimum angle
  of the face $F$.
\end{lemma}

\begin{proof}
  We can assume $v_1=(0,0,0)$ and $v_4=(0,0,h_3)$. Suppose
  $v_2=(a_2,b_2,c_2)$, $v_3=(a_3,b_3,c_3)$. Let $s,t$ such that
  \begin{eqnarray*}
    a_2s+b_2t&=&c_2\\
    a_3s+b_3t&=&c_3.
  \end{eqnarray*}
  It is clear that there exist such $s$ and $t$ since $v_1$, $v_2$,
  and $v_3$ do not lay on one line. Then the map $f(\tilde x)=B\tilde
  x$ with
  \[
  B=\left(\begin{array}{ccc}1&0&0\\0&1&0\\s&t&1\end{array}\right)
  \]
  sends $\widetilde P$ to $P$ where $\widetilde P$ is a prism with
  three vertical edges and some of its vertices are $\tilde
  v_1=(0,0,0), \tilde v_2=(a_2,b_2,0),\tilde v_3=(a_3,b_3,0)$ and
  $\tilde v_4=(0,0,h_3)$. Let $\widetilde F$ be the face $\tilde
  v_1\tilde v_2\tilde v_3$ of $\widetilde P$.

  Let $\tilde v$ be defined by $\tilde v(\tilde x)=v(x)$ if $x=B\tilde
  x$. Then we have
  \[
    \|v\|_{L^p(F)}^p=\frac1{\cos\gamma}
    \|\tilde v\|_{L^p(\widetilde F)}^p.
  \]
  Now, if $\widetilde Q$ is the right prism with vertices $\tilde
  v_1,\ldots,\tilde v_4$, $(a_2,b_2,h_3)$ and $(a_3,v_3,h_3)$, then we
  have using a trace inequality on $\widetilde Q$ and noting that
  $\widetilde Q\subseteq\widetilde P$ that
  \begin{align*}
    \|\tilde v\|_{L^p(\widetilde F)}^p &\le C_{p}h_3^{-1} \left(
      \|\tilde v\|_{L^p(\widetilde Q)}^p + h_3^p
      \|\tilde\partial_3\tilde v\|_{L^p(\widetilde Q)}^p \right)\\
    &\le C_{p}h_3^{-1} \left( \|\tilde v\|_{L^p(\widetilde P)}^p +
      h_3^p \|\tilde\partial_3\tilde v\|_{L^p(\widetilde P)}^p\right)
  \end{align*}
  with $C_{p}$ depending only on $p$.
  Therefore, we have
  \begin{align*}
    \|v\|_{L^p(F)}^p &= \frac{C_{p}}{\cos\gamma} h_3^{-1} \left(
      \|\tilde v\|_{L^p(\widetilde P)}^p + h_3^p
      \|\tilde\partial_3\tilde v\|_{L^p(\widetilde P)}^p \right) \\ &=
    \frac{C_{reg}}{\cos\gamma} h_3^{-1} |B| \left( \|v\|_{L^p(P)}^p +
      h_3^p \|\partial_3v\|_{L^p(P)}^p \right)
  \end{align*}
  where we used that $\tilde\partial_3 \tilde v(\tilde x) = \partial_3
  v(x)$. Since $|B|=1$ we obtain the desired result.
\end{proof}

\begin{lemma}\label{lem:trace_face}
  Let $T$ be an anisotropic element with the node $n$ on the singular
  edge and let $\sigma_{n}$ be a short edge. Let $P_{n}\subset
  \overline{S_T}$ be a parallelogram of maximal area having
  $\sigma_{n}$ as an edge and another edge on the singular edge, see
  Figure \ref{fig:FnPn}. And let $F_{n}$ the face of $S_T$ containing
  $P_{n}$. Then $|P_{n}|\ge C|F_{n}|$, and for all $v\in W^{1,1}(S_T)$
  we have
\begin{align*}
\|v\|_{L^1(P_{n})}  
& \le C |F_{n}||S_T|^{-1}
\left(\|v\|_{L^1(S_T)} +
|s_{1,T}|\|\partial_{s_{1,T}}v\|_{L^1(S_T)} +
|s_{2,T}|\|\partial_{s_{2,T}}v\|_{L^1(S_T)}\right).
\end{align*}
where $s_{1,T}$ and $s_{2,T}$ are two short edges of $T$.
\end{lemma}

\begin{figure}
  \begin{center}
    \includegraphics[height=6cm]{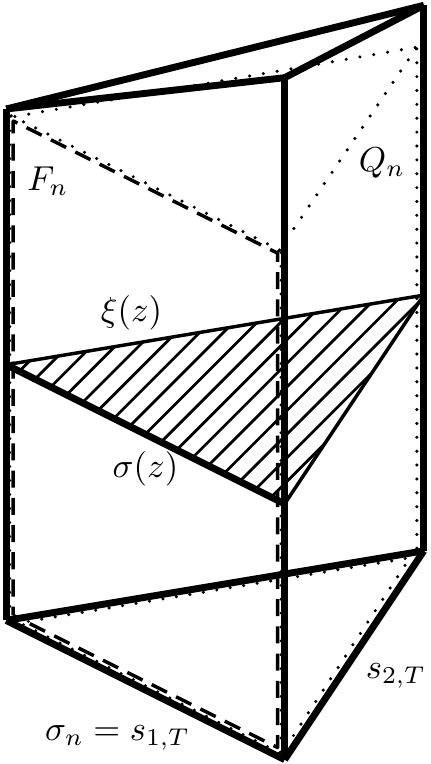}
  \end{center}
  \caption{\label{fig:FnPn}Illustration of the notation used in Lemma
    \ref{lem:trace_face}. The dotted lines indicate the prism $Q_n$,
    dashed lines the parallelogram $P_n$ while the triangle $\xi(z)$
    is hatched. Note that $\sigma(z)=\overline{\xi(z)}\cap F_n$.}
\end{figure}

\begin{proof}
  The inequality $|P_{n}|\ge C|F_{n}|$ follows from our
  assumptions on the mesh, in particular from the comparable length of
  opposite edges of $F_{n}$. For proving the estimate choose the
  coordinate system such that $n=(0,0,0)$.

  Assume first $v$ is regular. We have
  \begin{eqnarray*}
    \|v\|_{L^1(P_{n})} &\le& C
    \int_0^{h_{3,P_n}}\int_0^{|\sigma_{n}|}
    \left|v((0,0,z)+t\sigma_{n})\right|\,dt\,dz\\ &=&
    \int_0^{h_{3,P_n}} \int_{\sigma(z)}|v|\,ds\,dz
  \end{eqnarray*}
  where $\sigma(z)$ is the segment parallel to $\sigma_{n}$ and
  with the same length and passing through $(0,0,z)$. If $\xi(z)$ is 
  the triangle contained in $S_T$ having $\sigma(z)$ as an edge and
  being parallel to the bottom face of $S_T$, then since we can assume
  $v|_{\xi(z)}$ is regular (because $v$ is it), by a trace
  inequality we have
  \[
  \int_{\sigma(z)}|v| \le C \frac{|\sigma_{n}|}{|\xi|}\int_{\xi(z)}
  (|v|+|s_{1,T}||\partial_{s_{1,T}}v|+|s_{2,T}||\partial_{s_{2,T}}v|)
  \]
  where $|s_{1,T}|$ and $|s_{2,T}|$ are the lengths of two small edges
  of $T$ and $|\xi|=|\xi(0)|$. So we have
  \begin{eqnarray*}
    \|v\|_{L^1(P_{n})} &\le& C \frac{|\sigma_{n}|}{|\xi|}
    \int_0^{h_{3,P_n}} \int_{\xi(z)}
    (|v|+|s_{1,T}||\partial_{s_{1,T}}v|+|s_{2,T}||\partial_{s_{2,T}}v|)\\ &\le& C
    \frac{|F_{n}|}{|S_T|} \int_0^{h_{3,P_n}} \int_{\xi(z)}
    (|v|+|s_{1,T}||\partial_{s_{1,T}}v|+|s_{2,T}||\partial_{s_{2,T}}v|)\\ &\le & C
    \frac{|F_{n}|}{|S_T|} \int_{Q_{n}}
    (|v|+|s_{1,T}||\partial_{s_{1,T}}v|+|s_{2,T}||\partial_{s_{2,T}}v|) \\ &\le & C
    \frac{|F_{n}|}{|S_T|} \int_{S_T}
    (|v|+|s_{1,T}||\partial_{s_{1,T}}v|+|s_{2,T}||\partial_{s_{2,T}}v|)
  \end{eqnarray*}
  where $Q_{n}$ is the prism formed by the union of
  $\xi(z)$ with $z\in [0,h_{3,P_n}]$ that is contained in $S_T$.

  If $v\in W^{1,1}(S_T)$, let $\{v_k\}_k$ be a sequence of $C^\infty$
  functions converging to $v$ in $W^{1,1}(S_T)$. For each $k$ we have
  \begin{align*}
    \|v_k\|_{L^1(P_{n})} & \le C |F_{n}||S_T|^{-1} \left(\|v_k\|_{L^1(S_T)} +
      |s_{1,T}|\|\partial_{s_{1,T}}v_k\|_{L^1(S_T)} +
      |s_{2,T}|\|\partial_{s_{2,T}}v_k\|_{L^1(S_T)}\right).
  \end{align*}
  Now, the proof concludes by taking limit $k\to\infty$.
\end{proof}
\end{appendix}

\paragraph*{Acknowledgement.}
The work of all authors was supported by DFG (German Research
Foundation), IGDK 1754. The work of the second author is also
supported by ANPCyT (grant PICT 2010-1675 and PICTO 2008-00089) and by
CONICET (grant PIP 11220090100625). This support is gratefuly
acknowledged.

\bibliography{al}\bibliographystyle{plain}

\end{document}